\documentclass[letterpaper,12pt,openany,oneside]{article}

\usepackage[T1]{fontenc}
\usepackage[latin1]{inputenc}
\usepackage{epsfig}
\usepackage{tabstackengine}
\stackMath
\usepackage{amsmath,amssymb}
\usepackage[active]{srcltx}
\usepackage{dsfont}
\usepackage{systeme,mathtools}
\usepackage{xcolor,graphicx}
\usepackage{gensymb}
\usepackage{tikz}

\usepackage{gensymb}

\numberwithin{equation}{section}
\newtheorem{theoreme}{Theorem}[subsection]
\newtheorem{proposition}[theoreme]{Proposition}
\newtheorem{prop}[theoreme]{Proposition}
\newtheorem{remarque}[theoreme]{Remark}
\newtheorem{lemme}[theoreme]{Lemma}
\newtheorem{corollaire}[theoreme]{Corollary}
\newtheorem{conjecture}[theoreme]{Conjecture}
\newtheorem{definition}[theoreme]{Definition}
\newtheorem{exemple}[theoreme]{Example}
\newenvironment{proof}[1][Proof]{\noindent \textbf{#1.}~ }
{\hfill\rule{2mm}{2mm} \vspace{\parskip} }

\newcommand{\RR}{\ensuremath{\mathbb R}}
\newcommand{\Z}{\ensuremath{\mathbb Z}}
\newcommand{\R}{\ensuremath{\mathbb R}}
\newcommand{\PP}{\ensuremath{\mathbb P}}
\newcommand{\C}{\ensuremath{\mathbb C}}

\newcommand{\E}{\ensuremath{\mathbb E}}

\newcommand{\II}{\ensuremath{\mathcal I}}
\newcommand{\val}{\mathsf{val}}

\newcommand{\JJ}{\ensuremath{\mathcal J}}
\newcommand{\MM}{\ensuremath{\mathcal M}}

\newcommand{\N}{\ensuremath{\mathbb N}}

\newcommand{\ind}{\ensuremath{\mathds 1}}

\newcommand{\tr}{\ensuremath{\operatorname{tr}}}
\newcommand{\xx}{\ensuremath{\mathbf x}}
\newcommand{\yy}{\ensuremath{\mathbf y}}
\newcommand{\Id}{\ensuremath{\operatorname{Id}}}
\newcommand{\om}{\omega}
\newcommand{\Om}{\Omega}
\newcommand{\De}{\Delta}
\newcommand{\de}{\delta}
\newcommand{\La}{\Lambda}
\newcommand{\Ga}{\Gamma}
\newcommand{\la}{\lambda}
\newcommand{\ep}{\varepsilon}
\newcommand{\limn}{\lim_{n\to\infty}}
\newcommand{\al}{\alpha}
\newcommand{\be}{\beta}
\newcommand{\si}{\sigma}
\newcommand{\ga}{\gamma}
\newcommand*{\transp}[2][-3mu]{\ensuremath{\mskip1mu\prescript{\smash{\mathrm t\mkern#1}}{}{\mathstrut#2}}}%
\definecolor{darkblue}{rgb}{0,0,0.7} 
\newcommand{\darkblue}{\color{darkblue}} 
\newcommand{\defn}[1]{\emph{\darkblue #1}} 

\newcommand{\limla}{\displaystyle \lim_{\lambda \to 0}}

\title{Shapley-Snow kernels, multiparameter eigenvalue problems and stochastic games}
\author{Luc Attia\footnote{\'Ecole Polytechnique, Palaiseau, France.} \ and Miquel Oliu-Barton\footnote{Universit\'e Paris-Dauphine, PSL Research University, CNRS, CEREMADE, Paris, France.}
}

\date{April 14, 2019} 

\begin{document}
 \maketitle

\abstract{We establish, for the first time, a connection between stochastic games and multiparameter eigenvalue problems, using the theory developed by Shapley and Snow (1950). 
This connection provides new results, new proofs, and new tools for studying stochastic games.} 

\setcounter{tocdepth}{2}
\tableofcontents

\section{Introduction}
Stochastic games were introduced in the 1950's by Shapley \cite{shapley53}. They are played by two opponents over a finite set of states: the
 state variable follows a Markov chain controlled by both players, to each state corresponds a matrix game 
that determines a stage payoff, and  
the first player maximises a discounted sum of the stage payoffs 
 while the second player minimises the same amount.
 Stochastic games have a value 
 $v^k_\la\in \R$ for any discount factor $\la\in(0,1]$ and any initial state $k$. 
  Moreover, for each $\la$,
   the vector of values $v_\la=(v^1_\la,\dots,v^n_\la)\in \R^n$ is the unique fixed point of the so-called Shapley operator $\Phi(\la,\,\cdot\,):\R^n\to \R^n$, where $n\in \N^*$ denotes the number of states. 
The convergence of the values as $\la$ vanishes was established by Bewley and Kohlberg \cite{BK76} in the late 70's using Tarski-Seidenberg elimination theorem from mathematical logic and the Puiseux theorem. 
Three alternative proofs have been provided since then, by Szczechla, Connell, Filar and Vrieze \cite{SCFV97}, Oliu-Barton \cite{OB14} and Attia and Oliu-Barton \cite{AOB18a}. Besides the convergence, the latter provided a characterisation of the limit values. \\

But let us go back to matrix games. In the late 1920's, Von Neumann proved the celebrated minmax theorem: ``\emph{Every matrix game $G$ has a value, denoted by $\mathrm{val}(G)$, 
and both players have optimal strategies}''. 
The set of optimal strategies was characterised in the 1950's by Shapley and Snow \cite{SS50} 
as a polytope, and each of its extreme points 
corresponds to 
a square sub-matrix
 $\dot{G}$ of $G$, for which the following formula holds:
\begin{equation}
\label{formula_value}
\mathrm{val}(G)=\frac{\det(\dot{G})}{S(\mathrm{co}(\dot{G}))}
\end{equation}
For any matrix $M$, $\mathrm{co}(M)$ denotes its co-factor matrix and $S(M)$ denotes the sum of its entries. 
The sub-matrices characterising the extreme points of the set of optimal strategies 
are the so-called \emph{Shapley-Snow kernels} of the game. 
Szczechla, Connell, Filar and Vrieze \cite{SCFV97} noted, in the late 1990's, that applying the theory of Shapley and Snow to stochastic games provides, for any fixed discount factor $\la$, a system of $n$ polynomial equalities (in $n$ variables) that is satisfied by the vector of values $v_\la$.
Indeed, for each $1\leq k\leq n$ and $z\in \R^n$, the $k$-th coordinate of Shapley's operator $\Phi^k(\la,z)$ is the value of a matrix game,
denoted by $\mathcal{G}^k(\la,z)$, whose entries depend polynomially in $(\la,z)$.  By considering a Shapley-Snow kernel of each of these games at $z=v_\la$ and by setting: 
\begin{equation}\label{syst_1}
P^k(\la,z):=\det(\dot{\mathcal{G}}^k(\la,z))-z^k {S(\mathrm{co}(\dot{\mathcal{G}}^k(\la,z))}
\end{equation}
one deduces from \eqref{formula_value}  that $v_\la$ satisfies the polynomial equality $P^k(\la,v_\la)=0$. 
Although initially defined for the variables $(\la,z)\in(0,1]\times \R^n$, the polynomial system \begin{equation}\label{aas}P^1(\la,z)=\dots=P^n(\la,z)=0\end{equation}
can also be seen as 
an analytical variety in $\C^{n+1}$. To every choice of a square sub-matrix of $\mathcal{G}^k(\la,z)$ for each $1\leq k\leq n$ thus corresponds an analytical variety, so that their union $\mathcal{C}$ is an analytical variety too.
  Szczechla et al. \cite{SCFV97} prove that the set $\{(\la,v_\la),\, \la\in(0,1]\}$ is a regular $1$-dimensional connected component of $\mathcal{C}$, 
    from which they deduce the convergence of the values $v_\la$ as $\la$ vanishes.\\ 
    
 In the present paper, we propose to apply the theory of Shapley and Snow to stochastic games in a different manner, namely through multiparameter eigenvalue problems (MEP), a terminology introduced by Atkinson in the 1960's. As we will show, our approach considerably simplifies the analysis of stochastic games and provides several new results. 
 %
The connection between MEP, the theory of Shapley and Snow and stochastic games can be described as follows. 
 First of all, represent the stochastic game by an $n\times (n+1)$ array of matrices:
\begin{equation*}\label{aaas}D=\begin{pmatrix}
M_0^1 & M_1^1 &\dots & M^1_n\\
\vdots &  \vdots & \ddots  &  \vdots  \\
M_0^n & M_1^n &\dots &    M^n_n
\end{pmatrix}\end{equation*}
that contains all the relevant data of the stochastic game, namely, the matrices corresponding to each state, the transition probabilities and the discount factor. 
The array representation is reminiscent of MEP, 
except that the matrices in $D$ might be rectangular while MEP are only defined for arrays of square matrices. 
Indeed, for each $1\leq k\leq n$, the matrices 
$M^k_0,\dots,M^k_n$ are square matrices of equal size, then 
$D$ defines a MEP, that is, the problem of finding a vector $z\in \C^n$ which satisfies: 
\begin{equation}\label{syst_2}\det({M}^k_0+ z^1 {M}^k_1+\dots+z^n {M}^k_n)=0,\quad 1\leq k\leq n \end{equation}
MEP can be tacked by introducing $n+1$ auxiliary matrices, denoted by $\De^0,\dots,\De^n$, which allow to transform \eqref{syst_2} into the following uncoupled system: 
\begin{equation}\label{syst_3}\det(\De^k-z^k \De^0)=0,\quad 1\leq k\leq n \end{equation}
System \eqref{syst_3} is simpler to solve, as each variable appears in a separate equation. 
Moreover, Atkinson \cite{atkinson72} proved that \eqref{syst_2} and \eqref{syst_3} have the same solutions under suitable assumptions, such as the invertibility of the matrix $\De^0$. Applying the theory of MEP to stochastic games has two important consequences: on the one hand, it allows to transform the polynomial system \eqref{aas} satisfied by the vector $v_\la$ into an uncoupled a polynomial system, that is, a polynomial equation $P^k(\la,z^k)=0$ satisfied by $v^k_\la$, for each $1\leq k\leq n$; on the other hand, it provides new algebraic insight on the values. The bridge between MEP and stochastic games is provided by the theory of Shapley and Snow. 
Indeed, by considering a Shapley-Snow
kernel of each of the games $\mathcal{G}^k(\la,v_\la)$, $1\leq k\leq n$, like in Szczechla et al. \cite{SCFV97}, we restrict our attention to a $n\times (n+1)$ array $\dot{D}$ of square (and relevant) matrices. 

\subsection{Main results}
The combination of these two theories (Shapley and Snow \cite{SS50} and Atkinson \cite{atkinson72}), and their application to stochastic games is the main novelty of this paper, since our approach provides new tools, new results and simpler proofs of important known results. Our main results are the following: 
\paragraph{Result 1.} \emph{For any fixed $\la\in(0,1]$ and $1\leq k\leq n$ one has:
\begin{itemize}
\item[$(i)$] $v_\la^k$ is the unique $w\in \R$ satisfying $\mathrm{val}((-1)^n(\dot{\De}^k-w\dot{\De}^0))=0$.
\item[$(ii)$] $\mathrm{rank}(\dot{\De}^k-v^k_\la \dot{\De}^0) <  \max_{w \in \R} \mathrm{rank}(\dot{\De}^k-w\dot{\De}^0)$.
\item[$(iii)$] There exists a polynomial $P^k\in E^k$ such that $P^k(\la,v_\la^k)=0$.
\item[$(iv)$] If the stage payoffs, the transition probabilities and the discount factor are rational, 
then $v^k_\la$ is algebraic of order at most $\mathrm{rank}(\dot{\De}^0)$.
\end{itemize} }
In this result, $E^k$ denotes a finite set of bi-variate polynomials and $\dot{\De}^k$ and $\dot{\De}^0$ denote two square sub-matrices of $\De^k$ and $\De^0$ respectively. An explicit construction of $E^k$, $\dot{\De}^k$ and $\dot{\De}^0$ will be provided. Moreover, the bound of $(iv)$ is tight.

\paragraph{Result 2.}\emph{ For any fixed $1\leq k\leq n$ one has:
\begin{itemize}
\item[$(i)$]  Any accumulation point of $(v_\la^k)$, as $\la$ vanishes, belongs to a finite set $V^k$. As a consequence, the values converge and the limit $v^k_0:=\lim_{\la\to 0} v_\la^k$ belongs to $V^k$.
\item[$(ii)$]  If the stage payoffs, the transition probabilities and the discount factor are 
rational, then $v^k_0$ is algebraic of order at most $\mathrm{rank}(\dot{\De}^0)$.
\end{itemize}}
In this result, the set $V^k$ is constructed explicitly: it is the set of the real roots of finitely many polynomials obtained from the set of polynomials $E^k$. The set $V^k$ provides an alternative method to compute the value of $v^k_0$ exactly, provided that all the data of the game is rational. 

\paragraph{Result 3.} \emph{For any fixed $1\leq k\leq n$ one has: 
\begin{itemize}
\item[$(i)$] There exists $P^k\in E^k$  and $\la^k_0>0$ such that 
$P^k(\la,v_\la^k)=0$, for all $\la\in (0,\la^k_0)$.
\item[$(ii)$] There exists $\la^k_0>0$ and a finite set $W^k$ of Puiseux series on $(0,\la^k_0)$ such that the value function $\la\mapsto v^k_\la$, $\la\in(0,\la^k_0)$ belongs to $W^k$. 
\item[$(iii)$] As $\la$ vanishes one has: 
$|v^k_\la-v^k_0|=O(\la^{1/a})$, where $a=\mathrm{rank}(\dot{\De}^0)$. 
\end{itemize}}
In this result, the main novelty is the explicit construction of $E^k$ and $W^k$, from which one can deduce new upper bounds for the speed of convergence of the discounted values as the discount factor vanishes. Moreover, these results are obtained without invoking neither
Tarski-Seidenberg elimination principle nor the geometry of complex analytic varieties. 

\paragraph{Comments} 
\begin{itemize}
\item Result 1 refines the characterisation of $v^k_\la$ provided by the authors in \cite{AOB18a}, which states that ``\emph{$v^k_\la$ is the unique $w\in \R$ satisfying $\val((-1)^n(\De^k-w\De^0))=0$}''. 
\item In the sequel, stochastic games with state-dependent actions sets will be considered. If $K$ denotes some finite set of states and $p^k$ and $q^k$ denote the number of actions of Player 1 and 2, respectively, in state $k\in K$, then it is worth mentioning that $\dot{\De}^0$ is a square matrix whose size is bounded by $\prod_{k\in K} \min(p^k,q^k)$. Consequently, 
one has the following more explicit bound: 
$$\mathrm{rank}(\dot{\De}^0)\leq 
 \prod\nolimits_{k\in K} \min(p^k,q^k)$$ 
\end{itemize}

\subsection{Outline of the paper and notation}
The paper is organised as follows. Section 2 is devoted to a brief presentation of the three theories we are concerned with: stochastic games, the theory of Shapley and Snow and multiparameter eigenvalue problems.
For the former, we propose two presentations (a classical one, and a new one) and some relevant known results. 
Section 3 is devoted to establishing a link between the three theories, for a fixed discount factor, and to prove Result 1. 
In Section 4, we consider the case where the discount factor vanishes, and prove Results 2 and 3. Section 5 is devoted to some additional remarks concerning the tightness of the bounds, the exact computation of the limit values and an alternative construction of the so-called characterising polynomials. 
Section 6 is an Appendix. 

\paragraph{Notation.} 
The following notation will be used throughout the paper: 
\begin{itemize}
\item For any finite set $Z$, we denote its cardinality by $|Z|$ and $\De(Z)$ denotes the set of probability distributions over $Z$, i.e. $\{\al:Z\to [0,1],\ \sum_{z\in Z}\al(z)=1\}$. 
\item For any matrix $M$, $\transp{M}$ denotes its transpose, $S(M)$ denotes the sum of its entries by $S(M)$ and $\dot{M}$ denotes a sub-matrix of $M$. By $M\geq 0$ we indicate that all the entries of $M$ are nonnegative. 
\item For any \emph{square} matrix $M$, we denote its trace by $\tr(M):=\sum_i M_{ii}$ and its cofactor matrix by $\mathrm{co}( M)$. For each $(i,j)$, the 
$(i,j)$-th entry of $\mathrm{co}(M)$ is equal to $(-1)^{i+j}M_{ij}$ where $M_{ij}$ is the determinant of the matrix obtained by deleting the $i$-th row and $j$-th column. The matrices $M$ and $\mathrm{co}(M)$ are of same size and satisfy the following well-known formula: 
$$M \transp{\mathrm{co}(M)}=\det(M)\Id$$ where $\mathrm{co}(M)=1$ for any $1\times 1$-matrix $M$, by convention. 
\item We denote by $\textbf{1}$ and $U$, respectively, a column vector and a matrix of $1$'s. 
Their dimension will depend on the context. 
\item Any matrix $M$ is identified with a matrix game, also denoted by $M$. The value of $M$ is denoted by  $\mathrm{val}(M)$. Mixed strategies of both players are considered as \emph{column} vectors. 
For any couple of mixed strategies $(x,y)$, the expected payoff is given by $\transp{x} M y$. 

\item Suppose that $M$ is a $p\times q$-matrix, and that $(x,y)\in \R^p\times \R^q$. 
For any sub-matrix $\dot{M}$ of $M$, we denote by $\dot{x}$ and $\dot{y}$ the restrictions of $x$ and $y$ to the row and column indices of $\dot{M}$, respectively. 
\end{itemize}

\section{Preliminaries}
The aim of this section is to provide a brief presentation of 
stochastic games, the theory of Shapley and Snow and multiparameter eigenvalue problems.
For the former, we propose two presentations: the classical one, due to Shapley \cite{shapley53}, along with the results obtained therein; and a new presentation recently proposed by the authors in a previous work \cite{AOB18a}. Some examples will be provided, in order to help the reader getting acquainted with each of these theories.

\subsection{Standard stochastic games}\label{ssg}
Stochastic games are described by a tuple $\Ga^k_\la=(K,I,J,g,q,\la,k)$, where $K$ is the set of states, $I$ and $J$ are the action sets of Player 1 and 2 respectively, $g:K\times I\times J\to \R$ is the payoff function, $q:K\times I\times J\to \De(K)$ is the transition function, 
$\la\in(0,1]$ is a discount factor and $k\in K$ is an initial state.\\  

\noindent {We assume throughout the paper that $K$, $I$ and $J$ are finite sets}, and $K=\{1,\dots,n\}$.\\ 

\noindent The game $\Ga^k_\la$ is defined as follows. At every stage $m\geq 1$, knowing the current state $k_m$, the players choose simultaneously and independently actions $i_m\in I$ and $j_m\in J$. 
The triplet $(k_m,i_m,j_m)$ has two effects: it produces a stage payoff $g_m=g(k_m,i_m,j_m)$ and  determines the law $q(\,\cdot\,| k_m,i_m,j_m)$ of the state at stage $m+1$. Player $1$ maximises the expectation of $\sum_{m\geq 1}\la(1-\la)^{m-1} g_m$ given $k_1=k$, whereas Player $2$ minimizes the same amount. 
By Shapley \cite{shapley53}, this game has a value, denoted by $v^k_\la$. For both players, a strategy is a mapping from the set of finite histories into his own set of mixed actions. A strategy is optimal if it guarantees the value against any strategy of the opponent.\\ 


As already observed by Shapley \cite{shapley53}, the assumption that the current state is observed implies the existence of optimal stationary strategies, that is, strategies that depend only on the current state. 
For this reason, we will restrict our attention to stationary strategies throughout the paper. The set of stationary strategies are denoted, respectively, by $\De(I)^n$ and $\De(J)^n$. 
For any couple of stationary strategies $(x,y)$ and any initial state $k$, we denote by $\PP^{k}_{x,y}$ the unique
probability on the set of plays $(K\times I\times J)^\N$ induced by a couple $(x,y)$ on the $\si$-algebra generated by the cylinders. 
Similarly, we denote by $\E^k_{x,y}$ the expectation with respect to $\PP^k_{x,y}$. Finally, we denote by $\ga^k_\la(x,y)$ the corresponding expected payoff, i.e.:
\begin{equation}\label{expected_payoff} \ga^k_\la(x,y):=\E_{x,y}^{k}\left[\sum\nolimits_{m\ge 1} \la (1-\la)^{m-1}g(k_m,i_m,j_m)\right]\end{equation}
The following useful notions were introduced in \cite{shapley53}. 
\begin{definition}\label{local_game} For each $\la\in(0,1]$, $1\leq k\leq n$ and $z\in \R^n$, the \defn{local game} $\mathcal{G}^k(\la,z)$ is an $I\times J$-matrix game whose entries are given by: 
\begin{equation*}(\mathcal{G}^k(\la,z))^{ij} :=\la g(k,i,j)+(1-\la)\sum_{\ell=1}^n q(\ell|k,i,j)z^{\ell}\end{equation*}
\end{definition}
\begin{definition}\label{Shapley_operator} For each $\la\in(0,1]$, the \defn{Shapley operator} 
 $\Phi(\la,\,\cdot\,):\R^n\to \R^n$ is defined as follows. For each $1\leq k\leq n$ and $z\in \R^n$:
$$\Phi^k(\la,z):=\mathrm{val}(\mathcal{G}^k(\la,z))$$ 
\end{definition}
The main results of \cite{shapley53} can be stated as follows: 
\begin{enumerate}
\item
For each $\la\in(0,1]$ and $1\leq k\leq n$, the stochastic game $\Ga_\la^k$ has a value, denoted by $v^k_\la$, and both players have optimal stationary strategies. Moreover: 
\begin{equation*}
v_{\lambda}^k=\max_{x \in \De(I)^n} \min_{y\in \De(J)^n} \ga^k_\la(x,y)=\min_{y\in \De(J)^n} \max_{x \in \De(I)^n} \ga^k_\la(x,y)
\end{equation*}
\item For each $\la\in(0,1]$, the vector of values $v_\la\in \R^n$ is the unique fixed point of $\Phi(\la,\,\cdot\,)$, which is a strict contraction of $\R^n$ with respect to the $L^\infty$-norm, i.e. 
$\max_k|\Phi^k(\la,z)-\Phi^k(\la,\bar{z})|\leq (1-\la)\max_k |z^k-\bar{z}^k|$, for all $z,\bar{z}\in \R^n$.
\item For each $\la\in(0,1]$ and $1\leq k\leq n$, one has $|v^k_\la|\leq \|g\|:=\max_{(k,i,j)}|g(k,i,j)|$ and the map 
$\la\mapsto v^k_\la$ is $\|g\|$-Lipschitz continuous. 
\end{enumerate}

\subsection{A new presentation of stochastic games}\label{newpres}
It is customary to present stochastic games as a tuple $(K,I,J,g,q,\la,k)$, like we did 
in Section \ref{ssg}. Consider now an alternative presentation of the game as 
 the following $n\times (n+1)$ {array of matrices}: 
\begin{equation}\label{data} 
D(\la):=\begin{pmatrix}
\la G^1 & (1-\la)Q^1_1-U & (1-\la) Q^1_2& \dots &  (1-\la) Q^1_n \\
\la G^2 & (1-\la)Q^2_1 & (1-\la) Q^2_2 -U &\dots &  (1-\la) Q^2_n \\
\vdots &  \vdots & \vdots  &   \\
\la G^n  & (1-\la)Q^n_1 &(1-\la) Q^n_2 & \dots &(1-\la) Q^n_n- U
\end{pmatrix} \end{equation}
where for each $1\leq k,\ell \leq n$, we have set $Q^k_\ell$ and $G^k$ to be the following $|I|\times |J|$-matrices:
\begin{equation}\label{GQ}Q^k_\ell:=(q(\ell| k, i,j))_{i,j}\quad \text{ and }\quad G^k:=(g(k,i,j))_{i,j}\end{equation}
and where $U$ stands for a $|I|\times |J|$ matrix of ones. 
We will refer to $D(\la)$ as the \emph{data array}, as it carries all the data of the game, just like the tuple $(K,I,J,g,q,\la)$, where the initial state is not specified. 
Let us go one step further. For any $1\leq k\leq n$ and $0\leq \ell \leq n$, set: 
\begin{equation}\label{Ms} 
{M}^k_\ell:=\begin{cases} \la {G}^k & \text{ if }\ell=0\\ 
(1-\la){Q}^k_k-U & \text{ if }k=\ell\\ (1-\la){Q}^k_\ell & \text{ if }1\leq k\neq \ell \leq n  
\end{cases}\end{equation}
where the dependence on $\la$ has been omitted in order to simplify the notation. By doing so, 
the stochastic games $(K,I,J,g,q,\la,k)$, $1\leq k\leq n$ are now presented in the following form:
\begin{equation}\label{atk_SM}D(\la)=\begin{pmatrix}
M_0^1 & M_1^1 &\dots & M^1_n\\
\vdots &  \vdots & \ddots  &  \vdots  \\
M_0^n & M_1^n &\dots &    M^n_n
\end{pmatrix} \end{equation}
Note that, by construction, the array $D(\la)$ satisfies the following two properties: 
\begin{itemize}
\item[$(H1)$] For each $1\leq k\leq n$, the matrices 
$M^k_0,\dots,M^k_n$ are of same size
\item[$(H2)$] {For all $1\leq k, \ell \leq n$ and $k\neq \ell$ one has:}
$$M^k_k\leq 0,\quad M^k_\ell \geq 0\quad \text{ and }\quad M^k_1+\dots + M^k_n \leq -\la U$$
\end{itemize}
\begin{remarque}\label{rem0}
In our setting, all the $M^k_\ell$ are of same size, namely $|I|\times |J|$. 
However, for later purposes, it is more convenient to state the less restrictive property (H1), which corresponds to the situation where the sets of actions are state-dependent. 
\end{remarque}
To the array $D(\la)$ one can associate 
 $n+1$ auxiliary matrices, denoted by $\De^0,\dots,\De^n$. 
\begin{definition}\label{important_def1} 
For each $0\leq \ell \leq n$, let $D^{(\ell)}(\la)$ be the $n\times n$ array of matrices obtained by deleting the $(\ell+1)$-th column from $D(\la)$. Then, set: 
$$\De^\ell:=(-1)^\ell\det\nolimits_{\otimes} D^{(\ell)}(\la)$$ 
where $\det_\otimes$ stands for the Kronecker determinant. 
\end{definition}
The Kronecker determinant is very similar to the usual determinant except that 1) the usual product of scalars is replaced by the so-called Kronecker product of matrices and  2) rows and columns do not play symmetric roles. We refer the reader to the Appendix A for more details on Kronecker products and determinants.  
The matrices $\De^0,\dots,\De^n$ are well-defined thanks to $(H1)$, are of equal size\footnote{If 
$p^k\times q^k$ denotes the common size of $M^k_0,\dots,M^k_n$ then $\De^0,\dots,\De^n$ are of equal size $\prod_{k=1}^n p^k\times \prod_{k=1}^n q^k$. In the present context, $p^k=|I|$ and $q^k=|J|$ for all $k$ so that $\De^0,\dots,\De^n$ are $|I|^n\times |J|^n$-matrices.} and each of their entries depends polynomially on $\la$ of degree at most $n$.\\

Let us recall two useful results from \cite{AOB18a}. 
The following elementary lemma, which is a consequence of the diagonally dominant aspect of $(H2)$, will be used in the sequel. 
\begin{lemme}\label{positivite} All the entries of $(-1)^n {\De}^0$ are greater or equal than $\la^n$. 
\end{lemme}

\begin{remarque}For any couple of matrices $A$ and $B$ of equal size, the fact that all the entries of $B$ are nonzero and of same sign implies the existence of a unique $w\in \R$ such that $\val(A-wB)=0$. Thus, Lemma \ref{positivite} implies, in particular, that the equation $\mathrm{val}(\De^k-w\De^0)=0$ admits a unique solution.\end{remarque} 
The following result is the building stone of \cite{AOB18a} in obtaining a characterisation for the limit values. 
 None of the results of the present manuscript rely on this result; rather, a refinement is proposed in Theorem \ref{charRD}, stated as \emph{Result 1} in the introduction. 
\begin{theoreme}\label{charac1} Fix $\la\in(0,1]$ and $1\leq k\leq n$. Then $v_\la^k$ is the unique $w\in \R$ satisfying $\mathrm{val}((-1^n)(\De^k-w\De^0))=0$.
\end{theoreme}

\begin{remarque} 
The definition of the auxiliary matrices $\De^0,\De^1,\dots, \De^n$ was slightly different in \cite{AOB18a}. 
However, the two constructions coincide, up to a sign $(-1)^n$. 
\end{remarque}

\begin{exemple}\label{stand_ex} 
To illustrate the data array representation of a stochastic game, consider the following game introduced by Kohlberg \cite{kohlberg74}. 
Consider a stochastic game with $4$ states. 
States $3$ and $4$ are \emph{absorbing}, that is, once they are reached, these states are never left. The payoff in theses states is, respectively, $1$ and $-1$, regardless of the player's actions. By simplicity, we will assume that the players have only one action in these states. 
States $1$ and $2$ have action sets $I=\{T,B\}$ and $J=\{L,R\}$. The payoff functions are defined by 
$g(1,i,j)=\ind_{\{(i,j)=(T,L)\}}$ and $g(2,i,j)=-\ind_{\{(i,j)=(T,L)\}}$, and the  transitions, which are all deterministic, are described as follows: 
\begin{center}
\begin{tikzpicture}[xscale=1, yscale=1]
\draw[fill=red!0] (0,0) rectangle (1.6,1.6);
\draw[thin] (0,0.8)-- (1.6,0.8);
\draw[thin] (0.8,0)-- (.8,1.6);
    \node[scale=1] at (0.4,1.2) {\textcolor{black!100}{$1$}};
        \node[scale=1] at (1.2,1.2) {\textcolor{black!100}{$2$}};
    \node[scale=1] at (0.4,0.4) {\textcolor{black!100}{$2$}};
    \node[scale=1] at (1.2,0.4) {\textcolor{black!100}{$3$}};     
 \node [above] at (0.4,1.6) { \textcolor{black!100}{L}};
 \node [above] at (1.2,1.6) {\textcolor{black!100}{R}};
 \node [left] at (0,1.2) {\textcolor{black}{T}};
 \node [left] at (0,0.4) {\textcolor{black}{B}};

  \draw[fill=red!0] (4,0) rectangle (5.6,1.6);
\draw[thin] (4,0.8)-- (5.6,0.8);
\draw[thin] (4.8,0)-- (4.8,1.6);
    \node[scale=1] at (4.4,1.2) {\textcolor{black!100}{$2$}};
        \node[scale=1] at (5.2,1.2) {\textcolor{black!100}{$1$}};
    \node[scale=1] at (4.4,0.4) {\textcolor{black!100}{$1$}};
    \node[scale=1] at (5.2,0.4) {\textcolor{black!100}{$4$}};     
 \node [above] at (4.4,1.6) { \textcolor{black!100}{L}};
 \node [above] at (5.2,1.6) {\textcolor{black!100}{R}};
 \node [left] at (4,1.2) {\textcolor{black}{T}};
  \node [left] at (4,0.4) {\textcolor{black}{B}};
  \node  at (0.8,-0.5) {\textcolor{black}{1}};
    \node  at (4.8,-0.5) {\textcolor{black}{2}};
\end{tikzpicture}
\end{center}
where the numbers stand for states. In state $1$, for instance, both $(T,R)$ and $(B,L)$ lead to state $2$, whereas $(B,R)$ leads to state 1 and $(T,L)$ induces no transition. 
The corresponding array $D(\la)$ is given by: 
$$
\begin{pmatrix}
\begin{pmatrix}
\phantom{-} \la\phantom{-}   & \phantom{-}0\phantom{-} \\
0 & 0\end{pmatrix} & 
\begin{pmatrix}
-\la & \phantom{1}-1\phantom{-} \\
-1& -1\end{pmatrix} & 
\begin{pmatrix}
0& 1-\la \\
1- \la & 0\end{pmatrix} & 
\begin{pmatrix}
\phantom{-}0\phantom{-} &\phantom{-} 0\phantom{-} \\
0 &1- \la\end{pmatrix} & 
\begin{pmatrix}
\phantom{-}0\phantom{-} & \phantom{-}0\phantom{-} \\
0 & 0\end{pmatrix}\\ 

 \begin{pmatrix}
-\la  & 0 \\
\phantom{-}0\phantom{-} & \phantom{-}0\phantom{-}\end{pmatrix} & 
\begin{pmatrix}
0 & 1-\la \\
1-\la & 0\end{pmatrix} & 

\begin{pmatrix}
-\la & \phantom{1}-1\phantom{-} \\
-1& -1\end{pmatrix} & 
\begin{pmatrix}
\phantom{-}0\phantom{-} & \phantom{-}0\phantom{-} \\
\phantom{-}0\phantom{-} & \phantom{-}0\phantom{-}\end{pmatrix}& 
\begin{pmatrix}
\phantom{-}0\phantom{-} & \phantom{-}0\phantom{-} \\
0 & 1-\la\end{pmatrix}\\
\phantom{-}\la & 0 & 0& -\la & \phantom{-}0\\
-\la & 0 & 0& \phantom{-}0& -\la
\end{pmatrix}$$
where we have identified the $1\times 1$ matrices of the last two rows with scalars. 
The first auxiliary matrix is given by: 
\begin{eqnarray*}\label{de0}\De^0&=\la^2& \begin{pmatrix} \lambda^2 & \lambda& \lambda& \la(2 - \lambda)  \\ \lambda &\lambda &\la(2 - \lambda)&1 \\ \lambda & \la(2 - \lambda) &\lambda&1 \\ \la(2 - \lambda)&1&1&1  \end{pmatrix} 
 \end{eqnarray*}
 As states 1 and 2 are similar to each other, let us focus on state $1$. The auxiliary matrix $\De^1$ is given by: 
 \begin{eqnarray*}
\De^1&=\la^2&\begin{pmatrix} \la^2 & \la &-\la(1-\la) & 0\\
 \la & \la & 0 & -(1-\la)^2\\
-\la(1-\la) &0 & \la(1-\la) &1 \\
0 & -(1-\la)^2 &1-\la &1-\la \end{pmatrix}
 \end{eqnarray*}
Hence, for any $w\in \R$: 
$$\De^1-w\De^0=\la^2 
\begin{pmatrix} \lambda^2 (1-w)  & \lambda (1-w) &  -\la(1-\la)-\la w &  -(2\la-\la^2)w  \\
 \lambda (1-w) & \lambda (1-w) &  -(2\la- \la^2)w & -(1-\la)^2-w\\
-\la(1-\la)-\la w& -(2\la-\la^2)w & \la(1-\la)-\la w &1-\la -w \\
-(2\la-\la^2)w &-(1-\la)^2-w &1-\la -w&1-\la -w \end{pmatrix} $$
 We will come back to this example later on. \end{exemple}



%

\subsection{The theory of Shapley and Snow}\label{sskth}
The aim of this section is to briefly present the theory developed by Shapley and Snow \cite{SS50}. 
 Throughout this section, $G$ will denote a fixed $|I|\times |J|$-matrix game with value $v=\mathrm{val}(G)$. The set of optimal strategies for player $1$ and $2$ in $G$ will be denoted by $X^*\subset \De(I)$ and $Y^*\subset \De(J)$, respectively. These sets are compact, non-empty polytopes, so that they can be described by their (finitely many) extreme points. A characterisation of the extreme points of $X^*\times Y^*$, called \emph{basic solutions} of $G$, was the main result in Shapley and Snow \cite{SS50}. 
The following theorem is a convenient restatement of their results. 
 
\begin{theoreme}[Shapley and Snow 1950] \label{SSK} 
A couple $(x,y)\in X^*\times Y^*$ is a basic solution of $G$ if and only if there exists a square sub-game $\dot{G}$ satisfying: 
\begin{enumerate}
\item[$(1)$] $S(co(\dot{G})) \neq 0$ 
\item[$(2)$]
$\dot{x}= \frac{co(\dot{G})}{S(co(\dot{G}))} \dot{\mathbf{1}}$ and $\dot{y} = \frac{^t co(\dot{G})}{S(co(\dot{G}))} \dot{\mathbf{1}}$ 
\end{enumerate}
In this case, the following additional properties hold:  
\begin{enumerate}
\item[$(3)$]
$\mathrm{val}(G) = \mathrm{val}(\dot{G})=\frac{det(\dot{G})}{S(co(\dot{G}))}$
\item[$(4)$] $\transp{\dot{x}} \dot{G}=\mathrm{val}(G)\transp{\dot{\mathbf{1}}}$ and  $\dot{G}\dot{y}=\mathrm{val}(G)\dot{\mathbf{1}}$
\end{enumerate}
\end{theoreme}
\begin{remarque} Note that the $\dot{x}$ and $\dot{y}$ appearing in Theorem \ref{SSK} are strategies, as their components are nonnegative and add up to 1 (see Theorem \ref{SSK} (2)). Hence, $(x,y)$ and $(\dot{x},\dot{y})$ are equal, up to completing the latter with zeros. 
\end{remarque}
\begin{definition}
A \defn{Shapley-Snow kernel} (SSK) of $G$ 
is a square sub-matrix $\dot{G}$ satisfying the four conditions of Theorem \ref{SSK}, for some basic solution $(x,y)$. Let $\dot{I}\subset I$ and $\dot{J}\subset J$ be the subsets of actions that define the sub-matrix $\dot{G}$. 
\end{definition}
Theorem \ref{SSK} has many consequences. Among them, the next statement gathers those that will be used in the sequel. Its proof can be found in Section \ref{proofsssk}. 
For any vector $z\in \R^d$, we denote its span by $<z>:=\{tz,\,  t\in \R\}$. 
 \begin{proposition}\label{lemme_G-vU}
Let $\dot{G}$ be a Shapley-Snow kernel of $G$, corresponding to a basic solution $(x,y)$. Then: 
\begin{enumerate}
\item[$(i)$] $S(\mathrm{co}(\dot{G}-v\dot{U})) \neq 0$ 
\item[$(ii)$] $\det(\dot{G}-v\dot{U})=\mathrm{val}(\dot{G}-v\dot{U})=\mathrm{val}({G}-v{U})=0$
\item[$(iii)$] $\mathrm{Ker}(\dot{G}-v\dot{U})=<\dot{y}>$ and $\mathrm{Ker}(\transp{(\dot{G}-v\dot{U})})=<\dot{x}>$ 
\item[$(iv)$] $\mathrm{co}(\dot{G}-v\dot{U})=S(\mathrm{co}(\dot{G}-v\dot{U}))\, \dot{x}\transp{\dot{y}}$. 
\end{enumerate}

\end{proposition}
The following example illustrates the notion of a Shapley-Snow kernel. 
\begin{exemple} Consider the following $3\times 3$ matrix game:
$$G=\begin{pmatrix} 
1 & 0& 1 \\
0 & 1 & 2\\
3 & 2 & 0\end{pmatrix} $$
Clearly, Player $1$ does not have any pure optimal strategy so that none of the entries of $G$ is an SSK. Let us show that the following sub-matrix, obtained by setting $\dot{I}=\{2,3\}$ and $\dot{J}=\{1,3\}$ 
 is one:
$$\dot{G}=\begin{pmatrix} 
0& 2 \\
3 & 0\end{pmatrix} $$
By Theorem \ref{SSK} it is enough to check that $\dot{G}$ satisfies $S(co(\dot{G}))\neq 0$, and that completing $\dot{x}= \frac{co(\dot{G})}{S(co(\dot{G}))} \dot{\mathbf{1}}$ and $\dot{y} = \frac{^t co(\dot{G})}{S(co(\dot{G}))}\dot{\mathbf{1}}$ with zeros (outside $\dot{I}$ and $\dot{J}$, respectively) gives a pair of optimal strategies. 
An easy computation gives: 
$$\mathrm{co}(\dot{G})=\begin{pmatrix} 
\phantom{-}0& -3 \\
-2 & \phantom{-}0\end{pmatrix}, \quad  S(\mathrm{co}(\dot{G}))=-5, \quad \dot{x}=\left(\frac{3}{5},\frac{2}{5}\right), \quad 
\dot{y}=\left(\frac{2}{5},\frac{3}{5}\right)$$
A quick verification gives that, indeed, $x=(0,\frac{3}{5},\frac{2}{5})$ and $y=(\frac{2}{5},0,\frac{3}{5})$ are optimal strategies, so that $\dot{G}$ is an SSK corresponding to the basic solution $(x,y)$. The value of $G$ can be obtained using the formula of Theorem \ref{SSK} $(3)$:
$$\val(G)=\val(\dot{G})=\frac{\det(\dot{G})}{S(\mathrm{co}(\dot{G}))}=\frac{6}{5}$$ 

\end{exemple}

\subsection{Multiparameter eigenvalue problems}\label{sec_atkinson}
Consider an $n\times(n+1)$ array of real matrices: 
\begin{equation}\label{atk_SM_2}D=\begin{pmatrix}
M_0^1 & M_1^1 &\dots & M^1_n\\
\vdots &  \vdots & \ddots  &  \vdots  \\
M_0^n & M_1^n &\dots &    M^n_n
\end{pmatrix} \end{equation}
where for each $1\leq k\leq n$, the matrices 
$M^k_0,\dots,M^k_n$ are \emph{square matrices} of equal size. 


\noindent Multiparameter eigenvalue problems (MEP), a terminology introduced by Atkinson \cite{atkinson72}, is the problem of finding $z=(z^1,\dots,z^n)\in \C^{n}$ satisfying\footnote{It is worth mentioning that Atkinson \cite{atkinson72} considered the homogenous version of this problem, namely the problem of finding $(z^0,\dots,z^n)\in \C^{n+1}$ satisfying $\det(z^0 {M}^k_0+ \dots+z^n {M}^k_n)=0$  for all $1 \leq k\leq n$. Solutions to an homogeneous MEP are determined only up to a multiplicative factor. Moreover, there is a one-to-one map between the solutions to \eqref{M_syst} and solutions to the homogeneous MEP satisfying $z^0\neq 0$. For this reason, Atkinson's results can be easily transposed to the non-homogeneous case, more relevant for us.}: 
 \begin{equation}\label{M_syst} \left \{ 
\TABbinary\tabbedCenterstack[l]{
\det({M}^1_0+ z^1 {M}^1_1+\dots+z^n {M}^1_n)&=&0\\
   \vdots  &&\\
\det({M}^n_0+ z^1 {M}^n_1+\dots+z^n {M}^n_n)&=&0 }
\right.
\end{equation}
\begin{remarque} The array $D$ differs from the array representation $D(\la)$ of stochastic games in two aspects: 
\begin{itemize}
\item In addition to $(H1)$, all matrices in the array are supposed to be square matrices.
\item The array $D$ may or may not satisfy $(H2)$.
\end{itemize}
\end{remarque}
Each row $1\leq k\leq n$ of  \eqref{M_syst} defines polynomial in $n$ variables whose degree is the common size of $M^k_0,\dots, M^k_n$. 
In particular, 
when all the $M^k_\ell$ are of size $1\times 1$ (i.e. scalars) the 
system \eqref{M_syst} boils down to an affine system of equations. In this case, the auxiliary matrices are also scalars and 
 \eqref{M_syst} admits a unique solution if and only if $\De^0\neq 0$. When this is the case, the unique $z\in \R^n$ satisfying \eqref{M_syst} is given by $z^k=\De^k / \De^0$ for $1\leq k \leq n$, by Cramer's rule.\\ 

The extension of Cramer's rule to an arbitrary $n\times (n+1)$ array of matrices, due to Atkinson \cite{atkinson72}, relies on the $n+1$ auxiliary matrices $\De^0,\dots,\De^n$ of Definition \ref{important_def1}. Namely, introduce the so-called generalised MEP, which consists in finding $z\in \C^n$ satisfying:
 \begin{equation}\label{De_syst} \left \{ 
\TABbinary\tabbedCenterstack[l]{
\det({\Delta}^1- z^1 {\De}^0)&=&0\\
\,   \vdots  &&\\
\det({\Delta}^n- z^n {\De}^0)&=&0 }
\right.
\end{equation} 
Note that, unlike \eqref{M_syst}, 
where the unknown $z$ appears in every equation, in \eqref{De_syst} each coordinate of $z$ appears in a separate equation. In this sense, the latter system is an 
\emph{uncoupled system}, and thus much simpler to tackle. \\

Throughout this section, we denote by $S^M$ and $S^\De$ the set of solutions 
 of \eqref{M_syst} and \eqref{De_syst}, respectively.\\ 

One distinguishes between \emph{singular} and \emph{regular} MEP according to whether 
some of the polynomials $P^k(w):=\det({\Delta}^k- w {\De}^0)$ are identically zero or not. The case where $\De^0$ is invertible is the so-called nonsingular case, for which the problem can be easily solved. %
The following result can be found in Atkinson \cite[Chapter 6]{atkinson72}.


\begin{theoreme} If $\De^0$ is invertible, then $S^M=S^\De$.
\end{theoreme}
\begin{remarque} The set $S^M=S^\De$ can be easily described in this case. Indeed, the non-singularity of $\De^0$ implies that $
z\in S^\De$ if and only if $\det(\Delta^k(\De^0)^{-1}- z^k \Id )=0$ for all $1\leq k\leq n$, so that 
$S^M$ is entirely described by the set of eigenvalues of the matrices $\Delta^k(\De^0)^{-1}$, $1\leq k\leq n$. In particular, $S$ is a finite set 
and can be computed efficiently. \end{remarque} 


When some polynomial $P^k$ is identically zero, the problem is a singular MEP. This occurs, for instance, when $\mathrm{Ker}(\De^k)$ and  $\mathrm{Ker}(\De^0)$ share a non-zero vector. 
 The vacuous equality $P^k(w)=0$ is then replaced by the \emph{rank drop} condition: 
 $$\mathrm{rank}(\De^k - z^k \De^0) <  \max_{w \in \R} \mathrm{rank}(\De^k - w \De^0)$$ 
 The auxiliary system \eqref{De_syst} is thus replaced by the following one: 
 \begin{equation}\label{RP_syst} \left \{ 
\TABbinary\tabbedCenterstack[l]{
\mathrm{rank}(\De^1 - z^1 \De^0) &<&  \max_{w \in \R} \mathrm{rank}(\De^1 - w \De^0)\\
\,   \vdots  &&\\
\mathrm{rank}(\De^n - z^n \De^0) &<&  \max_{w \in \R} \mathrm{rank}(\De^n - w \De^0)}
\right.
\end{equation}
Let $S^R$ denote the set of solutions of \eqref{RP_syst}. This set refines the set $S^\De$ in the sense that the two coincide in the nonsingular case, but the inclusion $S^R\subset S^\De$ holds in general. 
 Muhi\v{c} and Plestenjak \cite[Theorems 3.5 and 3.7]
{MP09} establish the equality $S^M=S^R$ under the assumption that $n=2$ and that $S^M$ has finitely many solutions, each of which is algebraically and geometrically simple\footnote{See \cite{MP09} for more details.}. 
Instead, we will use the following result, proved in the Appendix (Section \ref{proof37}). 
 \begin{prop}\label{rank_drop}  Suppose that all the entries of $\De^0$ are nonzero and of same sign. 
Let $z\in S^M$. Suppose that for each $1\leq k\leq n$ there exists a couple of vectors $(x^k,y^k)$ satisfying: 
 $$\begin{cases} x^k\in \mathrm{Ker} \transp{(M^k_0+z^1 M^k_1+\dots+z^n M^k_n)}, &  x^k>0\\
y^k\in \mathrm{Ker}(M^k_0+z^1 M^k_1+\dots+z^n M^k_n), &  y^k> 0\end{cases}$$
Then $z\in S^R$. 
\end{prop}

\noindent The following example illustrates the relation between MEP and generalised MEP. 
\begin{exemple}
Consider the following $2\times 3$ array:
$$D=\begin{pmatrix} M^1_0 & M^1_1  & M^1_2 \\
M^2_0 & M^2_1  & M^2_2 \end{pmatrix}=
\begin{pmatrix} 2 & 1 & 1\\ \begin{pmatrix} 1 & 0\\ 0 & 1\end{pmatrix} & 
\begin{pmatrix} -1 & 0 \\ -1 & -1 \end{pmatrix} & \begin{pmatrix} 2 & 1 \\ 3 & 2\end{pmatrix}\end{pmatrix}$$
The associated MEP is the problem of finding $(u,w)\in \C^2$ satisfying:
$$\det\begin{pmatrix} 2 +u + w \end{pmatrix}=0,\quad \quad
\det\begin{pmatrix} 1-u +2w & w \\ -u+3w & 1-u+2w\end{pmatrix}=0$$ 
By definition, the auxiliary matrices $\De^0$, $\De^1$ and $\De^2$ are given by: 
$$\De^0=M^1_2\otimes M^2_1 - M^1_1\otimes M^2_2, \quad 
\De^1=-(M^1_0\otimes M^2_2 - M^1_2\otimes M^2_0), \quad 
\De^2=M^1_0\otimes M^2_1 - M^1_1\otimes M^2_0$$
The Kronecker product $A\otimes B$ coincides with the usual product when $A$ or $B$ (or both) are scalars so that one can easily compute: 
$$\De^0=\begin{pmatrix} 3 & 1 \\ 4 & 3\end{pmatrix}, \quad 
\De^1=\begin{pmatrix} -3 & -2 \\ -6 & -3 \end{pmatrix}, \quad 
\De^2=\begin{pmatrix}-3 & 0\\ -2 & -3 \end{pmatrix}$$
The generalised MEP consists then in finding $(u,w)\in \C^2$ satisfying: 
$$\det\begin{pmatrix} 3+3u & 2+u\\ 6+4u & 3+3u \end{pmatrix}=0,\quad \quad 
\det\begin{pmatrix} 3+3w & w \\ 2+4w & 3+3w\end{pmatrix}=0$$
These equalities determine two separate polynomial equations of degree $2$ which have roots $(u_1,u_2)$ and $(w_1,w_2)$, respectively. 
The matrix $\De^0$ being nonsingular, the MEP and the generalised MEP have the same solutions, i.e. 
$$S^M=S^R=S^\De=\{(u_1,w_1),(u_1,w_2),(u_2,w_1),(u_2,w_2)\}$$

\end{exemple}

\section{Stochastic games, SSK and MEP} \label{SG_MEP} 
\subsection{From stochastic games to MEP}\label{reduced}
Let $D(\la)$ be an $n\times (n+1)$ array representation of some stochastic game. 
Fix $1\leq k\leq n$.  By Shapley and Snow \cite{SS50}, the local game $\mathcal{G}^k(\la,v^k_\la)$ introduced in Definition \ref{local_game} 
admits a 
Shapley-Snow kernel, defined by some subsets of actions $\dot{I}^k\subset I$ and $\dot{J}^k\subset J$ satisfying $|\dot{I}^k|=|\dot{J}^k|$. By construction, for any $z\in\R^n$ one has: 
\begin{equation}\label{GM}\mathcal{G}^k(\la,z)-z^k U=M^k_0+ z^1 M^k_1+\dots + z^n M^k_n\end{equation}
Let $\dot{M}^k_0,\dot{M}^k_1,\dots,\dot{M}^k_n$ denote, respectively, the 
$\dot{I}^k\times \dot{J}^k$ sub-matrices of ${M}^k_0,{M}^k_1,\dots,{M}^k_n$, and let $\dot{D}(\la)$ 
be the corresponding $n\times (n+1)$ array of matrices, that is: 
 \begin{equation}\label{atk_SM_dot}\dot{D}(\la)=\begin{pmatrix}
\dot{M}_0^1 & \dot{M}_1^1 &\dots & \dot{M}^1_n\\
\vdots &  \vdots & \ddots  &  \vdots  \\
\dot{M}_0^n & \dot{M}_1^n &\dots &    \dot{M}^n_n
\end{pmatrix} \end{equation}
Let $\dot{\De}^0,\dot{\De}^1,\dots,\dot{\De}^n$ be the auxiliary matrices associated to $\dot{D}(\la)$ which, by construction, are sub-matrices of $\De^0,\De^1,\dots,\De^n$, respectively. Moreover, they are all square and of equal size $\prod_{k=1}^n |\dot{I}^k|$. Note that their dependence on $\la$, which is polynomial of degree at most $n$, and on the choice of the couples $(\dot{I}^k,\dot{J}^k)$, $1\leq k\leq n$ is omitted from the notation. \\ 



\noindent Consider now the following four systems in the variable $z\in \R^n$, where $1\leq k\leq n$:
$$\mathrm{val}(\dot{M}^k_0+ z^1 \dot{M}^k_1+\dots+z^n \dot{M}^k_n)=0$$
$$\det(\dot{M}^k_0+ z^1 \dot{M}^k_1+\dots+z^n \dot{M}^k_n)=0$$
$$\mathrm{rank}(\dot{\De}^k - z^k \dot{\De}^0) <  \max\nolimits_{w \in \R} \mathrm{rank}(\dot{\De}^k - w \dot{\De}^0)$$
$$\det(\dot{\De}^k- z^k \dot{\De}^0)=0$$

Let $T^{\dot{M}}$, 
$S^{\dot{M}}$, $S^{\dot{R}}$ and $S^{\dot{\De}}$ denote, respectively, the set of solutions of each these systems.  
Next paragraph is devoted to the relation between these subsets of $\R^n$. 
\subsection{From MEP to stochastic games}\label{reduced2}
Applying the theory of MEP to stochastic games one obtains the following result. 


\begin{prop}\label{SSK_atk} 
$\{v_\la\}= T^{\dot{M}}\subset S^{\dot{M}}$ and $\{v_\la\} \subset S^{\dot{R}}\subset S^{\dot{\De}}$. 
\end{prop}
\begin{proof} By Shapley \cite{shapley53}, the discounted value $v_\la\in \R^n$ is the unique solution to the system: 
\begin{equation}\label{loc_shap}\mathrm{val}(\mathcal{G}^k(\la,z)-z^k U)=0,\quad 1\leq k\leq n\end{equation}
Indeed, this system is a restatement of Shapley's fixed-point formulation $\Phi(\la,v_\la)=v_\la$ using the fact that $\mathrm{val}(M+w U)=\mathrm{val}(M)+w$ for any matrix $M$ and any $w\in \R$. Similarly, the system 
$$\mathrm{val}(\dot{M}^k_0+ z^1 \dot{M}^k_1+\dots+z^n \dot{M}^k_n)=0,\quad 1\leq k\leq n$$  
has a unique solution, namely the value of the stochastic game in which, for each $1\leq k\leq n$, the players are restricted to play actions in the set $\dot{I}^k\times \dot{J}^k$. Hence $T^{\dot{M}}$ is a singleton. Fix $1\leq k\leq n$ now.  The game $\mathcal{G}^k(\la,z)-z^k U$ has the same set of optimal strategies as the local game $\mathcal{G}^k(\la,z)$ for all $(\la,z)$ (see Lemma B2 in the Appendix). Therefore, $\dot{\mathcal{G}}^k(\la,z)-v^k_\la \dot{U}=\dot{M}^k_0+ v_\la^1 \dot{M}^k_1+\dots+v_\la^n \dot{M}^k_n$ is a Shapley-Snow kernel of $\mathcal{G}^k(\la,v_\la^k)-v_\la^k U$.  Consequently, by \eqref{loc_shap} and Proposition \ref{lemme_G-vU} $(ii)$ one has: 
\begin{eqnarray*}
0&=& \mathrm{val}(\mathcal{G}^k(\la,v_\la)-v^k_\la U)\\&=&\mathrm{val}(\dot{M}^k_0+ v_\la^1 \dot{M}^k_1+\dots+v_\la^n \dot{M}^k_n)\\& =&\det(\dot{M}^k_0+ v_\la^1 \dot{M}^k_1+\dots+v_\la^n \dot{M}^k_n)\end{eqnarray*} 
As these equalities hold for every $1\leq k\leq n$, it follows that $\{v_\la\}= T^{\dot{M}}\subset  S^{\dot{M}}$. 
The inclusion $\{v_\la\}\subset S^{\dot{R}}$ follows from Proposition \ref{rank_drop}, which can be applied since: first, all the entries of $\dot{\De}^0$ are non-zero and of same sign by Lemma \ref{positivite};  
second, the optimal strategies $(\dot{x}^k,\dot{y}^k)\in \De(\dot{I}^k)\times\De(\dot{J}^k)$ corresponding to the Shapley-Snow kernel $\mathcal{\dot{G}}^k(\la,v_\la)-v_\la^k U$ satisfy $\dot{x}^k>0$ and $\dot{y}^k>0$ because they are strategies; and third, by Proposition \ref{lemme_G-vU} $(iii)$ one has: 
$$\begin{cases}\dot{x}^k\in \mathrm{Ker}(\transp{(\dot{\mathcal{G}}^k(\la,v_\la)-v_\la^k\dot{U})})\\ 
\dot{y}^k\in \mathrm{Ker}(\dot{\mathcal{G}}^k(\la,v_\la)-v_\la^k\dot{U})\end{cases}$$
To prove the last inclusion $S^{\dot{R}}\subset S^{\dot{\De}}$, let $z\in S^{\dot{R}}$. By definition of $S^{\dot{R}}$, for all $1\leq k \leq n$ one has 
$\mathrm{rank}(\dot{\De}^k - z^k \dot{\De}^0) <  \max_{w \in \R} \mathrm{rank}(\dot{\De}^k - w \dot{\De}^0)$, so that $\dot{\De}^k - z^k \dot{\De}^0$ is not of full rank or, equivalently, $\det(\dot{\De}^k - z^k \dot{\De}^0)=0$. Consequently, $z\in S^{\dot{\De}}$.   
\end{proof}

  \subsection{The inclusions in Proposition \ref{SSK_atk}}\label{ex_2}
Let us illustrate the relations obtained in Proposition \ref{SSK_atk} via an easy example where $v_\la$ is not the unique element neither of $S^{\dot{M}}$ nor of $S^{\dot{R}}$, so that the inclusions are strict. 
\begin{remarque}
These strict inclusions also hold for Example \ref{stand_ex}, but the analysis is more intricate. For this reason, we have preferred to illustrate this particular point with another example. \end{remarque}

\begin{exemple} \label{ex_2}
  Consider the following \emph{absorbing} game: 
\begin{center}
\begin{tikzpicture}[xscale=1, yscale=1]
\draw[fill=red!0] (0,0) rectangle (1.6,1.6);
\draw[thin] (0,0.8)-- (1.6,0.8);
\draw[thin] (0.8,0)-- (.8,1.6);
    \node[scale=1] at (0.4,1.2) {\textcolor{black!100}{$1^*$}};
        \node[scale=1] at (1.2,1.2) {\textcolor{black!100}{$0$}};
    \node[scale=1] at (0.4,0.4) {\textcolor{black!100}{$0$}};
    \node[scale=1] at (1.2,0.4) {\textcolor{black!100}{$1^*$}};     
 \node [above] at (0.4,1.6) { \textcolor{black!100}{L}};
 \node [above] at (1.2,1.6) {\textcolor{black!100}{R}};
 \node [left] at (0,1.2) {\textcolor{black}{T}};
 \node [left] at (0,0.4) {\textcolor{black}{B}};
\end{tikzpicture}
\end{center}
where $*$ indicates an absorbing payoff. That is, the stage payoff is $0$ until $(T,L)$ or $(B,R)$ is played, in which case the stage payoffs are equal to $1$ forever after. This game can be represented by the following array: 
$$D(\la)=\begin{pmatrix} \begin{pmatrix} \la & 0\\0& \la \end{pmatrix} & 
\begin{pmatrix} -1 & -\la \\ -\la  & -1 \end{pmatrix} & \begin{pmatrix}1- \la & 0\\0& 1-\la \end{pmatrix} \\
\la & 0 & -\la \end{pmatrix}$$
The so-called ``normalised local games'' at $(u,w)\in\R^2$ are given by:
$$\begin{cases}\mathcal{G}^1(\la,(u,w))-uU=
\begin{pmatrix} \la+(1-\la)w-u & -\la u\\  -\la u & \la + (1-\la)w -u\end{pmatrix}\\ 
\mathcal{G}^2(\la,(u,w))-wU= \la(1-w) \end{cases}$$
Any square sub-matrix of $\mathcal{G}^1(\la,v_\la)$ is a possible candidate for being a Shapley-Snow kernel of this game, so that there are $5$ of them: the entire matrix and each of its entries. To see that the entire matrix is the unique SSK, we proceed as follows. Suppose that the top-left entry is a Shapley-Snow kernel, and define $\dot{D}(\la)$ and $S^{\dot{M}}$ accordingly. In this case, $(u,w)\in S^{\dot{M}}$ if and only if $\la+(1-\la)w-u=0$ and $\la(1-w)=0$, which has a unique solution $(1,1)$. But this cannot be the vector of values $v_\la$ because one has $v^1_\la\in(0,1)$ for each $\la\in(0,1]$; to see this, let Player $2$ choose $L$ and $R$ with equal probability and independently at every stage. A similar reasoning rules out the three other entries of the matrix, so that the unique kernel of the game is the entire matrix.\\ 

\noindent  As $\dot{D}(\la)=D(\la)$, one can omit the ``dots'' from the notation. By definition, 
$S^M$ is given by the following system of equations in the unknown $(u,w)\in \R^2$:
$$\begin{cases}\det\begin{pmatrix} \la+(1-\la)w-u & -\la u\\  -\la u & \la + (1-\la)w -u\end{pmatrix}&=0\\ 
\det(\la(1-w))&=0\end{cases}$$
Clearly, this system admits two solutions, namely $(\frac{1}{1+\la},1)$ and $(\frac{1}{1-\la},1)$. Consider now the auxiliary systems $S^R$ and $S^\De$. An easy calculation gives:
$$\De^0=\begin{pmatrix}\la &  \la^2 \\ \la^2 & \la\end{pmatrix},\quad \De^1=\begin{pmatrix}\la &  0 \\ 0 & \la\end{pmatrix} ,\quad \text{ and }\quad \De^2=\De^0$$
The last equality is in fact a more general property: \emph{for any absorbing state $k$ with payoff $g^k$ one has $\De^k=g^k \De^0$}. Note that $\De^0$ is invertible here, so that the corresponding MEP is nonsingular and, consequently, one has $S^R=S^\De$. To compute this set one solves the following uncoupled system:
$$\begin{cases}\det \begin{pmatrix} \la(1-u) & -\la^2u\\ -\la^2 u & \la(1- u)\end{pmatrix}&=0\\
\det((1-w)\De^0)&=0\end{cases}$$
which, again, has two solutions, $(\frac{1}{1+\la},1)$ and $(\frac{1}{1-\la},1)$. We have thus obtained: 
$$v_\la=\left(\frac{1}{1+\la},1\right)\in S^M=S^R=S^\De=\left\{\left(\frac{1}{1+\la},1\right), \left(\frac{1}{1-\la},1\right)\right\}$$
\end{exemple}

\subsection{A characterising polynomial}\label{characpoly}
In the previous example we showed that, for any $\la\in(0,1]$, the value $v_\la^1$ is one of the two real roots of the univariate polynomial:
$$P^1(u)=\det \begin{pmatrix} \la(1-u) & -\la^2u\\ -\la^2 u & \la(1- u)\end{pmatrix}=\la^2((1-u)^2-\la^2 u^2)$$
The next result states that this property holds in general as a consequence of 
Proposition \ref{SSK_atk}. 
\begin{prop}\label{char2} Fix $\la\in(0,1]$ and $1\leq k\leq n$. Then, there exists a polynomial $\dot{P}^k$ satisfying: 
$$\dot{P}^k(v_\la^k)=0,\quad \dot{P}^k\not\equiv 0 \quad \text{ and }\quad \mathrm{deg} \dot{P}^k \leq \mathrm{rank}(\dot{\De}^0)$$
\end{prop}
\begin{proof} 
By definition, the rank of $A$ is the size of the largest invertible square sub-matrix $\dot{A}$ of $A$. Similarly, for two matrices $A$ and $B$ of equal size, $\max_{w\in \R}\mathrm{rank}(A+wB)$ is the size of the largest square sub-matrix $\dot{A}+w\dot{B}$ of $A+wB$ such that the polynomial $w\mapsto \det(A+wB)$ is not identically $0$. Let $r^k:=\max_{w \in \R} \mathrm{rank}(\dot{\De}^k - w \dot{\De}^0)$. By the definition of the rank, there exists 
some $r^k\times r^k$ sub-matrix of $\dot{\De}^k - w \dot{\De}^0$ such that its determinant is a polynomial which is not identically $0$. Denote this polynomial by $\dot{P}^k$. 
The inclusion $v_\la\in S^{\dot{R}}$ obtained in Proposition \ref{SSK_atk} implies $\mathrm{rank}(\dot{\De}^k - v_\la^k \dot{\De}^0)< r^k$, 
so that $\dot{P}^k(v^k_\la)=0$. 
 The bound on the degree of $\dot{P}^k$ follows from \cite[Proposition 4.6]{demmel97} which states that, for any couple of  square matrices $A$ and $B$ of equal size, the polynomial $P(w):=\det(A+w B)$ is either 
identically $0$, or of degree $\mathrm{rank}(B)$. By definition, $\dot{P}^k$ is the determinant of some sub-matrix of $\dot{\De}^k-w\dot{\De}^0$ so that its degree is bounded by the rank of $\dot{\De}^0$. 
\end{proof}

\bigskip 
Determining the polynomial $P^k$ of Proposition \ref{char2} may be difficult, as it requires knowing the value $v_\la\in \R^n$, computing a Shapley-Snow kernel of each local game $\mathcal{G}^k(\la,v^k_\la)$, and then finding a sub-matrix of maximal rank. 
One way to overcome this difficulty is to note that $\dot{P}^k$ is the determinant of some square sub-matrix of $\dot{\De}^k-w\dot{\De}^0$, which is a square sub-matrix of ${\De}^k-w{\De}^0$. Hence, by considering the determinant of all possible square sub-matrices of ${\De}^k-w{\De}^0$ one obtains a finite family of polynomials containing $\dot{P}^k$. Among them, let $E^k$ denote the set of polynomials which are nonzero and degree at most $\mathrm{rank}(\dot{\De}^0)$. \\ 


Summing up, we have obtained the following results, which refine the characterisation of the values given by Theorem \ref{charac1}. 

\begin{theoreme}\label{charRD} Fix $\la\in(0,1]$ and $1\leq k\leq n$. Then:
\begin{itemize}
\item[$(i)$] $v_\la^k$ is the unique $w\in \R$ satisfying $\mathrm{val}(\dot{\De}^k-w\dot{\De}^0)=0$
\item[$(ii)$] $\mathrm{rank}(\dot{\De}^k-v^k_\la \dot{\De}^0) <  \max_{w \in \R} \mathrm{rank}(\dot{\De}^k-w\dot{\De}^0)$
\item[$(iii)$] There exists $P^k\in E^k$ such that $P^k(v_\la^k)=0$. 
\item[$(iv)$] If all the entries of $D(\la)$ are rational, $v^k_\la$ is algebraic of order at most $\mathrm{rank}(\dot{\De}^0)$
\end{itemize}
\end{theoreme}

\begin{remarque} \label{biv2} The novelty in $(iii)$ with respect to the so-called semi-algebraic approach is the identification of the finite set of polynomials $E^k$. Also, it is important to note that the polynomial satisfying $(iii)$ 
depends on $\la$ so that different polynomials may correspond to different discount factors. Finally note that, because $\De^k-w\De^0$ depends polynomially on $\la$, every coefficient of every polynomial in $E^k$ is a polynomial in $\la$, so that the elements of $E^k$ can be seen as bi-variate polynomials in $(\la,w)$. 
\end{remarque}

\subsection{Consequences: back to the examples}\label{exmlpe}
\emph{Back to Example \ref{ex_2}.}
Let $P^1(\la,u)$ denote the characterising polynomial for state $1$, seen as a bi-variate polynomial. Namely, 
$$P^1(\la,u)=\la^2((1-u)^2-\la^2 u^2)$$
Let us show that the polynomial equality satisfied by $v^1_\la$, namely $P^1(\la,v^1_\la)=0$, for every $\la\in(0,1]$, implies its convergence. First of all, note that $r:=\mathrm{deg}_\la P^1(\la,u)=4$, so that there exist unique univariate polynomials $P_0,\dots P_4$ such that:
$$P^1(\la,w)=P_0(u)+P_1(u)\la+P_2(u)\la^2+P_3(u)\la^3+P_4(u)\la^4$$
Namely, $P_0=P_1=P_3\equiv 0$, $P_2(u)=(1-u)^2$ and $P_4(u)=-u^2$. Hence, for any $u\in \R$: 
$$P^1(\la,u)=P_2(u)\la^2 + O(\la^4),\quad \text{ as }\la\to 0$$
Let us show that the polynomial $P_2$ and the term $O(\la^4)$ determine the limit value $v_0^1:=\lim_{\la\to 0} v^1_\la$, and the speed of convergence of $v_\la^1$ to $v^1_0$. 
Let $w_0$ be some accumulation point of $(v^1_\la)_\la$ along some vanishing sequence $(\la_m)$. Then: 
\begin{equation}\label{eqZ} \lim_{m\to +\infty}\frac{P^1(\la_m,v_{\la_m}^1)}{\la_m^{2}}=\lim_{m\to +\infty}P_2(w_0)+O(\la_m^2)=P_2(w_0)=0\end{equation}
Consequently, $w_0$ is a root of $P_2$. As this is true for any accumulation point, and $P_2$ has a unique root at $1$, one obtains 
$\lim_{\la\to 0}v^1_\la=1$.  Moreover, the relation
$$0=\frac{P^1(\la,v^1_\la)}{\la^2}=(1-v_\la^1)^2+O(\la^2),\quad \text{ as }\la \to 0$$
implies $|v_\la^1-1|=O(\la)$. 

\paragraph{Comments}
\begin{enumerate}\item The fact that $P_2$ has a unique root is not important. Indeed, suppose that $w_0<w_1$ are two different accumulation points. The continuity of $\la\mapsto v_\la^1$ implies then that every $w\in[w_0,w_1]$ is an accumulation point of $(v_\la^1)$. But, by \eqref{eqZ}, every accumulation point of $(v^1_\la)$ is a root of $P_2$. A contradiction, since by the choice of $P_2$, this polynomial is nonzero, and has thus finitely many roots. 
\item The bound on the convergence rate is not given by the degree, nor the subindex of $P_2$. Rather, it is given by the algebraic order of $v^1_0$ as a root of $P_2$.
\item The fact that the values converge is not a surprise, as the values converge for any stochastic game. Note, however, that this approach differs from all previously known proofs of convergence\footnote{In chronological order, the convergence of the values has been established by Bewley and Kohlberg \cite{BK76}, Szczechla, Connell, Filar and Vrieze \cite{SCFV97}, Oliu-Barton \cite{OB14} and Attia and Oliu-Barton \cite{AOB18a}.}.  A new proof of convergence, which extends the ideas exposed here, is provided in Section \ref{conv}.

\end{enumerate}

\bigskip \noindent For completeness, let us also illustrate these arguments on our standing example. 
\emph{Back to Example \ref{stand_ex}.}
The characterising polynomial of state $1$ is given by:
$$P^1(\la,u)=\det(\De^1-u\De^0)$$
For all $u\in \R$, this polynomial satisfies: 
$$P^1(\la,u)=16 u^2\la^{10} + O(\la^{11}),\quad \text{ as }\la\to 0$$ 
Therefore, the asymptotic behavior of $v^1_\la$ can be deduced from $P_{10}(u)=16u^2$ and $O(\la^{11})$. 
Like before, one obtains:
$$\begin{cases}v^1_0:= \lim_{\la\to 0} v^1_\la=0\\
|v_\la^1-v_0^1|=O(\la^{1/2})
\end{cases}$$

\subsection{The rank drop condition}
As the following example shows, $v^k_\la$ does not necessarily drop the rank of ${\De}^k-w{\De}^0$. Hence, the rank drop property of Theorem \ref{charRD} $(ii)$ gives a tighter characterisation for the values. 


\begin{exemple}
Consider the following array: 
$$ D=\begin{pmatrix} 1 & -\frac{3}{4} &  \frac{1}{4} \\
A&  \frac{1}{4} U &  -\frac{3}{4} U\end{pmatrix}
$$
where $A= \begin{pmatrix} \phantom{-}1 & -3 & -3 \\ -3 & \phantom{-}1 & -3 \end{pmatrix}$ and  
$U$ stands for a $2\times 3$ matrix of ones. It corresponds to a stochastic game with $2$ states and state-dependent actions sets: both players have one action in state 1, while in state 2 the players have 2 and 3 actions, respectively. More precisely, it is a specific instance of the array:
$$ D(\la)=\begin{pmatrix} \la G^1 & (1-\la)Q^1_1-U &  (1-\la)Q^1_2\\
 \la G^2 & (1-\la)Q^2_1 &  (1-\la)Q^2_2-U\end{pmatrix}
$$
for 
$\la=\frac{1}{2}$, $G^1=2$, $Q^1_1=Q^1_2=\frac{1}{2}$, $Q^2_1=Q_2^2=\frac{1}{2} U$ and 
$G^2=2 A$. \\

\noindent  
A straightforward calculation gives:
$$\De^0=\frac{1}{2} \begin{pmatrix} 1 & 1 &1 \\ 1 & 1 & 1\end{pmatrix}, \quad \De^1=\begin{pmatrix} 1 & 0 & 0 \\ 0 & 1 & 0 \end{pmatrix} \quad \text{ and }\quad \De^2=\begin{pmatrix} \phantom{-}1 & -2 & -2 \\ -2 & \phantom{-} 1 & -2 \end{pmatrix} $$
Consequently, for all $w\in \R$ one has:
$$\De^1-w\De^0=\begin{pmatrix} 1-\frac{w}{2}  & -\frac{w}{2}  &-\frac{w}{2}  \\ -\frac{w}{2} &1- \frac{w}{2} &-\frac{w}{2}  \end{pmatrix}
\quad \text{ and }\quad  
\De^2-w\De^0=\begin{pmatrix} 1-\frac{w}{2}  & -2 - \frac{w}{2}  &-2-\frac{w}{2}  \\ -2-\frac{w}{2} &1- \frac{w}{2} &-2-\frac{w}{2}  \end{pmatrix}
$$
Clearly,  
$\mathrm{rank}(\De^1-w\De^0)=\mathrm{rank}(\De^1-w\De^0)=2$ for all $w\in \R$, so that the rank drop condition is never satisfied. \\

Consider now the reduced array $\dot{D}$ constructed in Section \ref{SG_MEP}. Recall that, in order to compute a Shapley-Snow kernel for each local game, one needs to know the vector of values $v_\la=(v^1_\la,v^2_\la)$. To do so, first note that the third action is dominant for player $2$ in both games $\De^1-w\De^0$ and $\De^2-w\De^0$, for all $w\in \R$, so that:
$$\mathrm{val}(\De^1-w\De^0)=-\frac{w}{2}\quad \text{ and }\quad \mathrm{val}(\De^2-w\De^0)=-2-\frac{w}{2}$$
As $\mathrm{val}(\De^1-v^1_\la\De^0)=\mathrm{val}(\De^2-v_\la^2\De^0)=0$ by Theorem \ref{charac1}, it follows that $v_\la=(0,-4)$. 
The local game at state $1$ is a scalar, so it is trivially a Shapley-Snow kernel. At state $2$, the ``normalised local game'' $\mathcal{G}^2(\la,v_\la)-v^2_\la U$
is given by: 
$$A+ \frac{1}{4}Uv^1_\la-\frac{3}{4}Uv^2_\la=\begin{pmatrix} 4 & 0& 0  \\ 0&4 &0 \end{pmatrix}$$
which admits several Shapley-Snow kernels, the simplest being 
the scalar matrix $0$ corresponding to the top-right corner. By selecting this kernel one obtains the following reduced array (of scalars): 
$$\dot{D}=\begin{pmatrix} \phantom{-}1 & -\frac{3}{4} & \phantom{-} \frac{1}{4} \\
-3&  \phantom{-}\frac{1}{4} &  -\frac{3}{4}\end{pmatrix}
$$
As $\dot{D}$ is a real matrix, the auxiliary matrices $\dot{\De}^0,\dot{\De}^1$ and $\dot{\De}^2$ are scalars. A calculation gives $\dot{\De}^0=\frac{1}{2}$, $\dot{\De}^1=0$ and $\dot{\De}^2=-2$ so that 
$\dot{\De}^1-w\dot{\De}^0=-\frac{w}{2}$ and $\dot{\De}^2-w\dot{\De}^0=-2-\frac{w}{2}$. Hence, the rank drops at the values: 
$$\begin{cases} 0=\mathrm{rank}(\dot{\De}^1)<\max\nolimits_{w\in \R} \mathrm{rank}(\dot{\De}^1-w\dot{\De}^0)&=1\\
0=\mathrm{rank}(\dot{\De}^2+4\dot{\De}^0)<\max\nolimits_{w\in \R} \mathrm{rank}(\dot{\De}^2-w\dot{\De}^0)&=1\end{cases}$$
\end{exemple}

\section{Asymptotic behaviour of the values}\label{asym} 

The aim of this section is to derive from Theorem \ref{charRD} several consequences on the asymptotic behavior of the discounted values. In order to state our results with the best possible bounds, we will no longer assume that the action sets are state-independent. Rather, let $I^k\times J^k$ denote the action set at state $k$, for all $1\leq k\leq n$, and let the payoff function $g$ and transition function $q$ be defined over the set $$Z:=\{(k,i,j)\,|\, 1\leq k\leq n,\ (i,j)\in I^k\times J^k\}$$ 
All the results obtained so far can the extended word for word to the case of state-dependent action sets. 
Let $D(\la)$ be the array representation of the game, 
which satisfies the properties $(H1)$ and $(H2)$, and let $\De^0,\dots, \De^n$ denote the corresponding auxiliary matrices, which are of equal size $\prod_{k=1}^n |I^k|\times \prod_{k=1}^n |J^k|$. 


\paragraph{Notation.} Set $L:=\prod_{k=1}^n \min(|I^k|,|J^k|)$\\[0.2cm]
 Set $g^-:=\min_{(k,i,j)\in Z} g(k,i,j)$ and $g^+:=\max_{(k,i,j)\in Z} g(k,i,j)$ \\
 


Let $\dot{D}(\la)$ be the reduced array, and let $\dot{\De}^0,\dots,\dot{\De}^n$ be the corresponding auxiliary matrices, which are square matrices of size less or equal than $L$. Hence, in particular, $\mathrm{rank}(\dot{\De}^0)\leq \min(L,\mathrm{rank}(\De^0))$, a bound which does not require knowing the values, nor computing a Shapley-Snow kernel for each local game. 



\subsection{A new proof for the convergence of the values}\label{conv}
Let $1\leq k\leq n$ be fixed throughout this section. 
As already noted, see Remark \ref{biv2}, the polynomials in $E^k$ are bi-variate polynomials in $(\la,w)$. By construction, for all $P^k\in E^k$ one has: 
$$\begin{cases} \mathrm{deg}_\la P^k(\la,w)\leq Ln \\ 
\mathrm{deg}_w P^k(\la,w)\leq \mathrm{rank}(\dot{\De}^0)  \end{cases}$$
For each $P^k\in E^k$ there exists unique $0\leq s\leq Ln$ and a unique (uni-variate) polynomial $\varphi(P^k)\not\equiv 0$ satisfying the following relation for all $w\in \R$: 
$$P^k(\la,w)=\la^s \varphi(P^k)(w)+o(\la^s),\quad \text{ as }\la \to 0$$ 
Indeed, like we did in Section \ref{exmlpe}, 
let $r:= \mathrm{deg}_\la P^k(\la,w)$ and let $P_0,\dots,P_{r}$ be the unique univariate polynomials  satisfying $P^k(\la,w)=\sum_{\ell=0}^{r} P_\ell(w)\la^\ell$. Then, $\varphi(P^k)=P_s$, where $s$ is the smallest integer $m$ such that $P_m\not\equiv 0$.\\ 

Let $V^k$ denote the set of all roots of $\varphi(P^k)$ that lie on the interval $[g^-,g^+]$
as $P^k$ ranges over all polynomials of $E^k$. The following result formalises what was obtained in the two examples of Section \ref{exmlpe}. 


\begin{proposition}\label{finite_set} The limit $v^k_0:=\lim_{\la\to 0}v^k_\la$ exists. Moreover, $v^k_0\in V^k$. 
\end{proposition}
\begin{proof} Let $w_0$ be an accumulation point of $(v_\la^k)$ along the sequence $(\la_m)$, that is $\lim_{m\to+\infty} \la_m=0$ and $\lim_{m\to +\infty} v^k_{\la_n}=w_0$. Accumulation points exist because $v^k_\la\in [g^-,g^+]$ for all $\la\in(0,1]$. 
By Theorem \ref{charRD}, for each $m\geq 1$ there exists a polynomial $P^k_m\in E^k$ such that $P_m^k(\la_m,v_{\la_m}^k)=0$. 
The set $E^k$ being finite, up to extracting a sub-sequence we can assume that $P^k_m=P^k$ for all $m\geq 1$ and some fixed polynomial $P^k\in E^k$. Hence: 
\begin{equation*}\label{eq2}P^k(\la_m, v^k_{\la_m})=0,\quad \forall m\geq 1\end{equation*}
By the definition of $\varphi(P^k)$, for all $w\in \R$ one has:
$$P^k(\la,w)=\la^s \varphi(P^k)(w)+o(\la^s),\quad \text{ as }\la \to 0$$ 
Consequently, dividing by $\la^s$ and taking $\la$ to $0$ one obtains: 
$$0=\lim_{m\to +\infty}\frac{P^k(\la_m,v_{\la_m}^k)}{\la_m^s}= \varphi(P^k)(w_0)$$ 
Thus, any accumulation point of $(v^k_\la)$ belongs to $V^k$. Yet, $\la\mapsto v_\la^k$ is a real continuous function, and the set of accumulation points of a real continuous function is either a singleton or an interval. 
The finiteness of $V^k$ implies that it is necessarily a singleton, which gives the desired result. 

%
\end{proof}

The next result follows directly from Proposition \ref{finite_set}, and the fact that, for any $P^k\in E^k$ one has 
$\mathrm{deg} \, \varphi(P^k)\leq \mathrm{rank}(\dot{\De}^0)$. 
\begin{corollaire}\label{alg2} Suppose that the payoff function $g$ and the transition function $q$ take only rational values. Then $v_0^k$ is algebraic of degree at most $\mathrm{rank}(\dot{\De}^0)$.\end{corollaire} 




\subsection{Speed of convergence and Puiseux expansion}
Fix $1\leq k\leq n$ throughout this section. 
Let us start by recalling the definition of a Puiseux series. 
A map $f:(0,\ep)\to \C$ is a \defn{Puiseux series} if there exists $N\in \N^*$, $m_0\in \Z$ and a sequence $(b_m)_{m\geq m_0}$ in $\C$ such that: 
$$f(\la)=\sum_{m\geq m_0}b_m \la^{m/N}$$
Any bounded Puiseux series satisfies $m_0\geq 0$ so that, in particular, it converges as $\la$ vanishes. The following result, due to Puiseux \cite{puiseux1850}, will be referred as the \emph{Puiseux theorem}. \emph{For any bi-variate polynomial $P(\la,w)$ satisfying $\mathrm{deg}_w P\geq 1$, there exists $\la_0>0$ such that the roots of $P(\la,\,\cdot\,)$ are Puiseux 
series in the interval $(0,\la_0)$}.\\

By the Puiseux theorem, the set of roots of all polynomials $P^k\in E^k$ satisfying $\mathrm{deg}_w P\geq 1$, is a finite set of 
Puiseux series. Let this set of series be denoted by $W^k$. 
 Our next result follows directly from Theorem \ref{charRD} $(iii)$ and the Puiseux theorem. 
\begin{proposition}\label{boundpuiseux} The following assertions hold: 
\begin{itemize}
\item[$(i)$] There exists $P^k\in E^k$  and $\la^k_0>0$ satisfying: $P^k(\la,v_\la^k)=0$, for all $\la\in (0,\la^k_0)$.
\item[$(ii)$] There exists $\la^k_0$ such that $\la\mapsto v_\la^k$ belongs to $W^k$ on $(0,\la^k_0)$. 
\item[$(iii)$] As $\la$ vanishes one has:  
$|v^k_\la-v^k_0|=O(\la^{1/a})$, where $a=\mathrm{rank}(\dot{\De}^0)$. 
\end{itemize}
\end{proposition}
\begin{remarque} The main novelty of $(i)$ and $(ii)$ is the explicit construction of $E^k$ and $W^k$, and the fact that 
we use directly the Puiseux theorem, that is, without invoking Tarski-Seidenberg elimination principle.
Concerning $(iii)$, not only this bound is sharper than all previously obtained bounds, there are also good reasons to expect it to be tight (see Section \ref{tight}). 
\end{remarque}

\begin{proof} 
$(i)$ and $(ii)$ 
By finiteness of the set $E^k$, there exists a common interval $(0,\ep)$ where all the Puiseux series of $W^k$ are well-defined.
%
By Theorem \ref{charRD} $(iii)$, for each $\la\in(0,1]$ there exists $P^k\in {E}^k$ satisfying  
$P^k(\la,v_\la^k)=0$. Consequently, for any $\la\in(0,\ep)$, the point $(\la,v_\la^k)\in \R^2$ lies on the graph of one of the Puiseux series in $W^k$. 
The continuity of $\la\mapsto v^k_\la$ implies that, as $\la$ goes to $0$, $v^k_\la$ may change from one Puiseux series to another only at points where two series intersect. As two different Puiseux series cannot intersect infinitely many times on $(0,\ep)$, so that there exist some $0<\ep'<\ep$ such that any two Puiseux series 
are either congruent or disjoint in $(0,\ep')$. Consequently, $\la\mapsto v^k_\la$ is one of them on $(0,\ep')$, which proves $(i)$ and $(ii)$, for $\la^k_0:=\ep'$. \\ 
$(iii)$ Let $P^k\in E^k$ and $\la^k_0$ be given by $(i)$. 
By definition,  $\varphi(P^k)$ is a uni-variate polynomial whose degree is bounded by $\mathrm{deg}_w P^k$.  
 Moreover, there exists $0\leq s\leq Ln$ 
and $t\geq 1$ such that, for all $w\in \R$:
$$P^k(\la,w)=\la^s \varphi(P^k)(w)+O(\la^{s+t}),\quad \text{ as }\la \to 0$$
Since one also has $P^k(\la,v_\la^k)=0$ for all $\la\in(0,\la^k_0)$, it follows that  
$\varphi(P^k)(v^k_\la)=O(\la^{t})$ as $\la\to 0$. Thus, in particular one has $\varphi(P^k)(v_0^k)=0$. Consequently, there exists an integer $1\leq b\leq \mathrm{deg}\, \varphi(P^k)\leq \mathrm{rank}(\dot{\De}^0)$ and a polynomial $R^k$ such that $R^k(v^k_0)\neq 0$ and: 
$$\varphi(P^k)(w)=(w-v^k_0)^b R^k(w),\quad \forall w\in \R$$
Hence, taking $w=v^k_\la$ one has: 
\begin{eqnarray*}
0=P^k(\la,v^k_\la)=\la^s (v_\la^k-v^k_0)^b R^k(v_\la^k)+ O(\la^{s+t}),\quad \text{ as }\la \to 0\end{eqnarray*}
which implies $|v^k_\la-v_0^k|= O(\la^{t/b})$ for $\la$ close to $0$. 
  The result follows, as $t/b$ is minimal for $t=1$ and $b=\mathrm{rank}(\dot{\De}^0)$. 
\end{proof}




\section{Concluding remarks}
\subsection{Tightness of the bounds}\label{tight}
\paragraph{Simple stochastic games.}
A \emph{simple stochastic game} is one satisfying 
$$\min(|I^k|,|J^k|)=1,\quad \text{ for all } 1\leq k\leq n$$
In particular, Markov decision processes are simple stochastic games, as they can be modeled as a stochastic game where $|J^k|=1$ for all $1\leq k\leq n$. For these games, one has $L=\prod_{k=1}^n\min(|I^k|,|J^k|)=1$ so that by Proposition \ref{boundpuiseux}, $v_\la^k$ converges to $v^k_0$ at a rate $O(\la)$ and there exist polynomials $a^k_0(\la)$ and $a^k_1(\la)$ of degree at most $n$ such that $v_\la^k$ is a root of 
$a^k_1(\la)w+ a^k_0(\la)=0$ for all sufficiently small $\la$. 

\paragraph{Absorbing games. }
An \emph{absorbing game} is one satisfying $|I^k|=|J^k|=1$ for all $2\leq k\leq n$, as one can assume with no loss of generality that state $1$ is the unique non-absorbing state and that both players have one action at every other state. Hence, $L=\min(|I^1|,|J^1|)$. 
By Proposition \ref{boundpuiseux} the characterising polynomial $P^1(\la,w)$ of $v_\la^1$ is of degree at most $L$ in $w$. The following example, due to Kohlberg \cite{kohlberg74} shows that this bound is tight. 
\begin{exemple}\em For any $p \geq 1$, consider the following absorbing game of size $p\times p$ introduced by Kohlberg \cite{kohlberg74}:
$$
\begin{pmatrix}
1^* & 0^* & \dots & 0^*\\
0 & 1^* & \ddots & \vdots \\
\vdots & \ddots & \ddots & 0^*\\
0 & \dots & 0 & 1^*
\end{pmatrix}$$
where $c^*$ indicates a stage payoff of $c$ and a certain transition to an absorbing state with payoff $c$. 
For every $\la\in(0,1)$, the entire matrix is the unique Shapley-Snow kernel of $\mathcal{G}^1(\la,v^1_\la)$ so that the characterising polynomial for $v^1_\la$ is given by the following equation:
$$P^1(\la,w)=\det\begin{pmatrix} 1-w & -w &\dots & -w \\
-\la w & 1-w&  & \vdots\\
\vdots & & \ddots & -w\\
-\la w & \dots & -\la w & 1-w\end{pmatrix}=(1-w)^p+\la R(w)+o(\la) $$
for some univariate polynomial $R$ satisfying $R(1)\neq 0$. From the equality $P^1(\la,v^1_\la)=0$ one deduces: 
$$v_\la^1=1-\la^{1/p} (R(v_\la^1))^{1/p}+o(\la^{1/p})$$
Hence, $v^1_\la$ converges to $1$ at a rate $\la^{1/p}$. As $p=L$, this example hits the bound of Corollary \ref{alg2}, namely $ |v^1_\lambda - v^1_{0}| = O(\lambda^{\frac{1}{L}})$. 
\end{exemple}
\paragraph{The general case.} Hansen et al. \cite{HKMT11} proved that, for a game with state-independent action sets $I$ and $J$ of common size $m$ and rational data, the algebraic degree of $v^k_\la$ (and, similarly, of the limit $v^k_0$) is bounded by $(2m+5)^n$, the best known bound so far. An example is also provided in \cite{HKMT11} of a game with $n+1$ states satisfying $|I^1|=|J^1|=1$ and $|I^k|=|J^k|=m$ for $2\leq k \leq n+1$, and where the algebraic degree of the discounted values is $m^n$. Note that $m^n$ coincides with $L:=\prod_{k=1}^{n+1}\min(|I^k|,|J^k|)$ in this example. Hence, there exists a family of stochastic games of arbitrary size, both in states and in actions, such that the algebraic degree of $v_\la^k$ is $L$. In this sense, Theorem \ref{charRD} $(iv)$ 
provides a tight bound for the algebraic degrees of $v^k_\la$. Because the algebraic degree of $v_\la^k$, the algebraic degree of $v^k_0$ and the speed of convergence of $v_\la^k$ to $v_0^k$ are closely related to each other, it is natural to think that the bounds we have obtained for the latter are tight too. However, we have not been able to establish these results. 


\subsection{Computing the exact values}\label{algo}
Fix $1\leq k\leq n$. By Proposition \ref{finite_set}, the limit value $v^k_0$ belongs to $V^k$, which is a set of roots of finitely many polynomials. The finiteness of this set was crucial in determining a new proof for the convergence of the values. We argue here that the set $V^k$ can also be used for algorithmic purposes, namely, if we are looking for the \emph{exact value} of $v_0^k$ in the case where all the entries of $g$ and $q$ are rational. 

\paragraph{An efficient algorithm.}
For an algebraic number $w\in \R$ its \emph{minimal polynomial} is the unique monic polynomial of least degree satisfying $P(w)=0$. By Kannan, Lenstra and Lovasz \cite{KLL88}, there exists an algorithm (referred in the sequel as the KLL algorithm) that computes the minimal polynomial $P$ of $w$, given a bound on the degree of $P$, a bound on the bit-size\footnote{For any integer $p\in \Z$, its bit-size is given by $\mathrm{bit}(p):=\log_2(\lfloor p\rfloor)+1$. For any rational number $p/q$ one defines $\mathrm{bit}(p/q):=\mathrm{bit}(p)+\mathrm{bit}(q)$.} of the coefficients of $P$ and an $\ep$-approximation of $w$, for $\ep$ small enough. A precise upper bound for $\ep$ is provided in \cite{KLL88}, as a function of the bounds on the coefficients and the degree of the minimal polynomial.\\ 

As already noted by one of the authors \cite{OB19}, the approach proposed in this manuscript yields to the best known bounds concerning the algebraic degree of $v^k_0$ and the coefficients of its minimal polynomial. Plugging them into the KLL algorithm, together with some approximation $v^k_\la$ yields a method for computing the exact values of $v^k_0$. A precise upper bound for $\la$ so that $v^k_\la$ is a good enough approximation of $v^k_0$ in the KLL algorithm, is obtained in \cite[Proposition 4.1]{OB19}. 

\paragraph{An alternative algorithm.}
Consider now the following alternative method for computing $v^k_0$ exactly. First, compute the \emph{separation} of $V^k$, that is:
$$\de:=\min\{|w-w'|\,|\,w,w'\in V^k,\ w'\neq w\}$$
Second, compute a $\de/2$-approximation for $v^k_0$, that is, $v^k_\la$ for an appropriate $\la$ such that 
$|v_\la^k-v^0|\leq \de/ 2$, where a precise explicit expression for $\la$ is given in \cite[Proposition 4.1]{OB19}. 
By the choice of $\de$ one clearly has: 
$$\left[v^k_\la-\frac{\de}{2},v^k_\la+\frac{\de}{2}\right]\cap V^k =\{v^k_0\}$$ 
Hence, the exact value of $v^k_0$ is obtained. 
\begin{remarque}Though the computation of $V^k$ may be problematic for large games, the alternative algorithm has the advantage of being easy and self-contained. \end{remarque}


\subsection{Another characterising polynomial}\label{altpoly}
For a given $\la\in(0,1]$ and $1\leq k\leq n$, we proved in Section \ref{characpoly} the existence of a characterising polynomial for $v^k_\la$, that is, one that satisfies $P^k(\la,v_\la^k)=0$. 
Our construction requires two steps: first, we define a reduced matrix game $\dot{\De}^k-w\dot{\De}^0$ by taking 
a Shapley-Snow kernel of each local game $\mathcal{G}^k(\la,v_\la)$; second, we use the rank drop condition 
of the values for this game. 
 In this paragraph, we propose the following different method for determining a characterising polynomial, based on 1) the theory of Shapley and Snow, but this time applied to the game $\De^k-w\De^0$ directly, and 2) the equality $\val((-1)^n(\De^k-v_\la^k\De^0))=0$ established by the authors in \cite{AOB18a} (see Theorem \ref{charac1}).  \\ 

Fix $1\leq k\leq n$. 
 Set $\mathcal{I}:=I^1\times \dots\times I^n$ and $\mathcal{J}:=J^1\times \dots\times J^n$. The game $\De^k-v_\la^k \De^0$ is a $\mathcal{I}\times \mathcal{J}$-matrix game, so that it admits a Shapley-Snow kernel. Let 
$\overline{\mathcal{I}}^k\subset \mathcal{I}$ and $\overline{\mathcal{J}}^k\subset \mathcal{J}$ be the subsets of actions that define one of its Shapely-Snow kernels. 
Let $\overline{\De}^0$ and $\overline{\De}^k$ be the $\overline{\mathcal{I}}^k\times \overline{\mathcal{J}}^k$ sub-matrices of $\De^0$ and $\De^k$, respectively, and for any $w\in \R$ set:
$$\begin{cases}{G}^k(w):= (-1)^n({\De}^k-w{\De}^0)\\ 
\overline{G}^k(w):= (-1)^n(\overline{\De}^k-w\overline{\De}^0)\\ 
\overline{P}^k(w):=\det(\overline{\De}^k-w\overline{\De}^0)\end{cases}$$
The polynomial thus obtained is another characterising polynomial of $v_\la^k$.

\begin{prop}\label{char3} 
 The polynomial $\overline{P}^k$ satisfies: 
$$\overline{P}^k(v_\la^k)=0,\quad \overline{P}^k\not\equiv 0,\quad  \text{ and }\quad \mathrm{deg} \overline{P}^k = \mathrm{rank}(\overline{\De}^0)$$
\end{prop}

\begin{proof} $(i)$ By Theorem \ref{charac1}, one has $\val(G^k(v_\la^k))=0$. By Proposition \ref{lemme_G-vU} $(iii)$ one then has: 
\begin{equation}\label{three}\val(G^k(v_\la^k))=\val(\overline{G}^k(v_\la^k))=\det(\overline{G}^k(v_\la^k))=0
\end{equation}
Therefore, $\overline{P}^k(v_\la^k)=0$. 
To prove $\overline{P}^k\not\equiv 0$ it is enough to show that its derivative $(\overline{P}^k)'$ 
is not identically $0$. 
%
For any two square matrices $M$ and $H$ of same size, by Jacobi's formula one has: $$\det(M+\ep H)=\det(M)+\ep \tr(\transp{\mathrm{co}(M)}H)+o(\ep),\quad \text{ as } \ep\to 0$$
Equivalently, the directional derivative of $\det(M)$ in the direction $H$ is $\tr(\transp{\mathrm{co}(M)}H)$. Applying this result to $M=\overline{\De}^k-w\overline{\De}^0$ and $H=-\overline{\De}^0$ one obtains:
$$(\overline{P}^k)'(v_\la^k)= \tr(\transp{\mathrm{co}(\overline{G}^k(v^k_\la))}(-\overline{\De}^0))$$
By Proposition \ref{lemme_G-vU} $(i)$ and $(iv)$, the matrix $\mathrm{co}(\overline{G}^k(v^k_\la))$ is not identically zero and has all its entries of same sign. Similarly, by Lemma \ref{positivite}, all the entries of $-\overline{\De}^0$ are non-zero and of same sign. Consequently, 
$(\overline{P}^k)'(v_\la^k)\neq 0$, so that $\dot{P}^k\not\equiv 0$. 
To obtain the degree of $\overline{P}^k$, we proceed like in the proof of Proposition \ref{char2}. By 
\cite[Proposition 4.6]{demmel97}, for any square matrices $A$ and $B$ the polynomial $P(w):=\det(A+w B)$ is either identically $0$ or of degree $\mathrm{rank}(B)$. 
\end{proof}

\begin{remarque}
The bound on the degree of $\dot{P}^k$ is considerably better than the bound we obtained for $\overline{P}^k$. Indeed, one has:
$$\begin{cases}
\mathrm{deg}\dot{P}^k\leq \mathrm{rank}(\dot{\De}^0)\leq \prod_{k=1}^n \min(|I^k|,|J^k|)\\
\mathrm{deg}\overline{P}^k=\mathrm{rank}(\overline{\De}^0)\leq \min(\prod_{k=1}^n|I^k|, \prod_{k=1}^n |J^k|)\end{cases}$$
\end{remarque}

\begin{remarque} We have exhibited two different constructions that lead to a characterising polynomial for $v^k_\la$, that is: either we consider a Shapley-Snow Kernel of the game $(-1)^n(\De^k-v_\la^k \De^0)$ and use the fact that this game has value $0$, or we consider a sub-matrix of maximal rank of the reduced game $\dot{\De}^k-v_\la^k\dot{\De}^0$ obtained by taking a Shapley-Snow kernel at each local game. If the following condition holds:  
$$S(\mathrm{co}(\dot{\De}^k-v_\la^k\dot{\De}^0)) \neq 0$$ 
then $(-1)^n(\dot{\De}^k-v_\la^k \dot{\De}^0)$ is a Shapley-Snow kernel of $(-1)^n({\De}^k-v_\la^k{\De}^0)$, 
in which case the same polynomial can be obtained with the two constructions. 
\end{remarque}


\section{Appendix}
\section*{Appendix A: Kronecker products} \label{kro}
Let us start by recalling the definition of the Kronecker product of two matrices and of the Kronecker determinant of an array of matrices. 
\paragraph{Definition A1.} 
 The  \defn{Kronecker product} of two matrices $A$ and $B$ of sizes $m\times n$ and $p\times q$ respectively, denoted by $A\otimes B$, is an $mp \times nq$ matrix defined by blocks as follows: 
\[A \otimes B = \begin{pmatrix} a_{11} B & \cdots & a_{1n}B \\ \vdots & \ddots & \vdots \\ a_{m1} B & \cdots & a_{mn} B \end{pmatrix}\]
\paragraph{Definition A2.} 
The \defn{Kronecker determinant} of an $n\times n$ array of matrices: 
$$ \begin{pmatrix}  A^1_1& \dots & A^1_n\\ \vdots & \ddots & \vdots \\ A^n_n & \dots & A^n_n\end{pmatrix}$$
is well-defined if and only for each $1\leq k\leq n$, the matrices $A^k_1,\dots,A^k_n$ are of same size. In this case, it is given by: 
$$
\det\nolimits_\otimes\begin{pmatrix}  A^1_1& \dots & A^1_n\\ \vdots & \ddots & \vdots \\ A^n_n & \dots & A^n_n\end{pmatrix}:=
 \displaystyle\sum_{\sigma\in \Sigma(n)} \epsilon(\sigma) A^1_{\sigma(1)} \otimes \cdots \otimes A^n_{\sigma(n)} 
$$
where $\Sigma(n)$ is the set of permutations of $\{1,\dots,n\}$ and $\epsilon(\si)$ is the signature of $\si$. 
\paragraph{Properties A3.} The following well-known properties have been used in this manuscript: 
\begin{itemize}
 \item[$(K1)$] The Kronecker product $\otimes$ is bilinear and associative, but not commutative. 
 \item[$(K2)$] Let $A_1,\dots,A_n$ and $B_1,\dots,B_n$ be some matrices such that the products $A_kB_k$ are well-defined. Then
$(A_1 \otimes \cdots \otimes A_n) (B_1 \otimes \cdots \otimes B_n) = (A_1 B_1) \otimes \cdots \otimes (A_n B_n)$.
 \item[$(K3)$] The Kronecker determinant $\det_\otimes$ has similar properties as the usual determinant, that is: it is multilinear and alternating, but only with respect to the columns. Indeed, because of the non-commutativity of the Kronecker product, rows and columns do not play the same role, and the determinant needs to be developed by columns. 
\end{itemize}
\medskip 

In order to express the last two properties we need to introduce the \emph{canonical bijection} mapping the product set $\{1,\dots, p_1\}\times \dots\times \{1,\dots, p_n\}$ into the set $\{1,\dots, \prod_{\ell=1}^n p_\ell\}$ using the lexicographical order. That is, for any $p_1,\dots,p_n\in \N^*$ set:
 \begin{eqnarray*}
\{1,\dots,p_1\}\times \dots\times \{1,\dots,p_n\} &\rightarrow &\{1,\dots, \prod\nolimits_{\ell=1}^n p_\ell\}\\ 
(i_1,\dots, i_n) 
& \mapsto & (i_1-1) C_1+\dots+ (i_n-1)C_n+1
\end{eqnarray*}
where  $C_\ell:=\prod_{r:=\ell+1}^n p_r$ for each $1\leq \ell< n$ and $C_n=1$. 

\medskip 
\begin{itemize} 
\item[$(K4)$] Let $A_k$ be a $p_k\times q_k$, for all $1\leq k\leq n$. 
 Let $r$ and $s$ be, respectively, the images of $(i_1,\dots,i_n)\in \{1,\dots,p_1\}\times \dots\times \{1,\dots, p_n\}$ and 
 $(j_1,\dots,j_n)\in \{1,\dots,q_1\}\times \dots\times \{1,\dots, q_n\}$ with respect to the canonical bijection. Then the entry $(r,s)$ of $A_1 \otimes \cdots \otimes A_n$ is given by: 
$$(A_1 \otimes \cdots \otimes A_n)^{r s}= A_{1}^{i_1j_1}\cdots  A_{n}^{i_n j_n}$$

\item[$(K5)$] Let $A_{11},\dots, A_{nn}$ be an $n\times n$ array of matrices such that for all $1\leq k\leq n$, the matrices on the $k$-th row $A_{k1},\dots,A_{kn}$ are of same size $p_k\times q_k$. 
Let $r$ and $s$ be, respectively, the images of $(i_1,\dots,i_n)\in \{1,\dots,p_1\}\times \dots\times \{1,\dots, p_n\}$ and 
 $(j_1,\dots,j_n)\in \{1,\dots,q_1\}\times \dots\times \{1,\dots, q_n\}$ with respect to the canonical bijection. 
Then the entry $(r,s)$ of the matrix $\det_\otimes (A_{11},\dots, A_{nn})$ satisfies:
$$\det\nolimits_\otimes \begin{pmatrix}  A_{11}& \dots & A_{1n}\\ \vdots & \ddots & \vdots \\ A_{n1} & \dots & A_{nn}\end{pmatrix}^{r s}= \det \begin{pmatrix}  A^{i_1 j_1}_{11}& \dots & A^{i_1j_1}_{1n}\\ \vdots & \ddots & \vdots \\ A^{i_nj_n}_{n1} & \dots & A^{i_nj_n}_{nn}\end{pmatrix} $$
\end{itemize}

%

\section*{Appendix B: Proof of Proposition \ref{lemme_G-vU}}\label{proofsssk}
Let us start by recall its statement (see Section \ref{sskth}) 

\paragraph{Proposition \ref{lemme_G-vU}.} \emph{
Let $\dot{G}$ be a Shapley-Snow kernel of $G$, corresponding to a basic solution $(x,y)$. Then: 
\begin{enumerate}
\item[$(i)$] $S(\mathrm{co}(\dot{G}-v\dot{U})) \neq 0$ 
\item[$(ii)$] $\det(\dot{G}-v\dot{U})=\mathrm{val}(\dot{G}-v\dot{U})=\mathrm{val}({G}-v{U})=0$
\item[$(iii)$] $\mathrm{Ker}(\dot{G}-v\dot{U})=<\dot{y}>$ and $\mathrm{Ker}(\transp{(\dot{G}-v\dot{U})})=<\dot{x}>$ 
\item[$(iv)$] $\mathrm{co}(\dot{G}-v\dot{U})=S(\mathrm{co}(\dot{G}-v\dot{U}))\, \dot{x}\transp{\dot{y}}$. 
\end{enumerate}}

\bigskip 
\noindent The proof is based on three easy lemmas. 
\paragraph{Lemma B1.} \emph{
For any square matrix $M$ one has:
\begin{itemize}
\item[$(i)$] $\det(M+wU)=\det(M)+w S(\mathrm{co}(M))$, for all $w\in \R$
\item[$(ii)$] The map $w\mapsto S(\mathrm{co}(M+w U))$ is constant
\item[$(iii)$] The maps $w\mapsto \mathrm{co}(M+wU)\textbf{1}$ and 
$z\mapsto \transp{\mathrm{co}(M+wU)}\textbf{1}$ are constant\\
\end{itemize} } 

\begin{proof} $(i)$ Let $M$ be some square matrix. The function $w\mapsto \det(M+w U)$ is a polynomial in $w$. Subtracting one row from all other rows of $M+wU$, it is clear that its degree is at most $1$. From the formulae $\tr(\transp{M}U)=S(M)$ and from Jacobi's formula: 
$$\det(M+\ep H)=\det(M)+\ep\tr(\transp{\mathrm{co}(M)}H)+o(\ep), \quad \text{ as } \ep \to 0$$
which hold for any square matrix $M$, one deduces $\frac{\partial}{\partial w}\det(M+w U)(0)=S(\mathrm{co}(M))$. Hence, $\det(M+w U)=\det(M)+ w S(\mathrm{co}(M))$ for any square matrix $M$ and any $w\in \R$.\\
$(ii)$ 
Applying $(i)$ to $M+wU$ and $-w$ yields:
$$\det(M)=\det((M+wU)-wU)=\det(M+wU)-w S(\mathrm{co}(M+wU))$$
Comparing with $(i)$, one obtains $S(\mathrm{co}(M))=S(\mathrm{co}(M+wU))$ for any $M$ and $w$.\\
$(iii)$ By the symmetric role of both players, it is enough to prove the first statement. Let $m\in\N^*$ be the size of $M$, and let $M_1,\dots,M_m$ be its rows. Then, for each $\ell=1,\dots,m$ the $\ell$-th component of the vector $\mathrm{co}(M)\textbf{1}$ satisfies:
\begin{eqnarray*}
(\mathrm{co}(M)\textbf{1})^\ell&=&\det(M_1,\dots,M_{\ell-1},\textbf{1},M_{\ell+1},\dots,M_m)\\
&=& \det(M_1+w\textbf{1},\dots,M_{\ell-1}+w\textbf{1},\textbf{1},M_{\ell+1}+w\textbf{1},\dots,M_m+w\textbf{1})\\
&=& (\mathrm{co}(M+wU)\textbf{1})^\ell
\end{eqnarray*}
where the second equality follows from the properties of the determinant, as we added $w$ times the $\ell$-th column to the other columns. 
\end{proof}



\paragraph{Lemma B2.} \emph{
 Let $\dot{G}$ be a Shapley-Snow kernel for $G$. Then $\dot{G}+w\dot{U}$ is a Shapley-Snow kernel of the translated game $G+wU$, for any $w\in \R$.}\\  

\begin{proof} 
To prove that $\dot{G}+w\dot{U}$ is a Shapley-Snow kernel, it is enough to check properties $(1)$ and $(2)$ of Theorem \ref{SSK}. On the one hand,
 $S(\mathrm{co}(\dot{G}+w\dot{U}))=S(\mathrm{co}(\dot{G}))\neq 0$ by Lemma B1 $(ii)$. On the other hand, it follows from Lemma B1 $(iii)$ that the following strategies do not depend on $w$:  
$$\dot{x}(w)=  \frac{\mathrm{co}(\dot{G}+w\dot{U})}{S(\mathrm{co}(\dot{G}+w\dot{U}))}\mathbf{1},  \quad
\dot{y}(w)= \frac{\transp{\mathrm{co}}(\dot{G}+w\dot{U})}{S(\mathrm{co}(\dot{G}+w\dot{U}))}\mathbf{1} $$
which completes the proof.
\end{proof}

\paragraph{Lemma B3.}\emph{
 Let $M$ be a square matrix of size $a\in \N^*$ and rank $a-1$, and let $x$ and $y$ be such that $\mathrm{Ker}(\transp{M})=<x>$ and $\mathrm{Ker}(M)=<y>$. Then
there exists a constant $\al\neq 0$ such that 
$\mathrm{co}(M)=\al \, x \transp{y}$.} \\

\begin{proof}
Using the relation $\transp{A} \ \mathrm{co}(A)  = \det(A) \Id$, which is valid for any matrix $A$, and $\det(M) = 0$, we get $\transp{M} \mathrm{co}(M)  = 0$.  Moreover since $\mathrm{Ker}(M) = <y>$, all the rows of $\mathrm{co}(M)$ are proportional to $y$. Hence, there exists $x'$ such that $\mathrm{co}(M) = x' \transp{y}$. This equality shows that the columns of $\mathrm{co}(M)$  are proportional to $x'$, and a symmetric argument gives that the columns are proportional to $x$. Therefore, $x$ and $x'$ are proportional. Let $\al\in \R$ be such that $x'=\al x$ so that $\mathrm{co}(M) = \alpha x \transp{y}$. As $M$ is of rank $n-1$, the matrix $\mathrm{co}(M)$ is non-zero, so that $\al\neq 0$, which proves the result. 
 \end{proof}



\paragraph{Proof of Proposition \ref{lemme_G-vU}.} 
 $(i)$ By Lemma B2, $\dot{G}-v\dot{U}$ is a Shapley-Snow kernel for $G-vU$ so that, in particular,
$S(\mathrm{co}(\dot{G}-v\dot{U}))\neq 0$. \\[0.2cm] 
$(ii)$ 
For any matrix $M$ and $z\in \R$, clearly $\mathrm{val}(M+zU)=\mathrm{val}(M)+z$. Hence, the formulae of Theorem \ref{SSK} $(3)$ yield:
$$0=\mathrm{val}({G}-vU)=\mathrm{val}(\dot{G}-v\dot{U})=\frac{\det(\dot{G}-v\dot{U})}{S(\mathrm{co}(\dot{G}-v\dot{U}))}$$
$(iii)$ By the symmetric role of both players, it is enough to prove the first statement. The matrix $\dot{G}-v\dot{U}$ is not invertible by $(ii)$, and its matrix of cofactors $\mathrm{co}(\dot{G}-v\dot{U})$ is non-zero, thanks to $(i)$. Hence, if $1\leq b\leq \min(|I|,|J|)$ denotes its size, one has $\mathrm{rank}(\dot{G}-v\dot{U})=b-1$ or, equivalently $\mathrm{dim}(\mathrm{Ker}(\dot{G}-v\dot{U}))=1$. 
Yet
by Theorem \ref{SSK} $(4)$, one has $\dot{G} \dot{y} =v \dot{\textbf{1}}$ and $\dot{y}\neq 0$,
so that $(\dot{G}-v \dot{U})\dot{y}=0$. Consequently  $\mathrm{Ker}(\dot{G}-v\dot{U})=<\dot{y}>$.  \\[0.2cm] 
$(iv)$ It follows from Lemma B3, as the hypotheses are satisfied thanks to $(iii)$. Hence, $\mathrm{co}(\dot{G}-v\dot{U})=\al\, \dot{x}\transp{\dot{y}}$ for some $\al\neq 0$, so that $S(\mathrm{co}(\dot{G}-v\dot{U}))= S(\al\, \dot{x}\transp{\dot{y}})=\al$ because $S(\dot{x}\transp{\dot{y}})=1$. $\hfill \blacksquare$


\section*{Appendix C: Proof Proposition \ref{rank_drop}}\label{proof37} 

Let us start by recalling the statement this result (see Section \ref{sec_atkinson}). Recall that an $n\times (n+1)$ array of matrices $D=(M^k_\ell)$ is given, where for all $1\leq k\leq n$ the matrices $M^k_0,\dots,M^k_n$ are square and of equal size. 

\paragraph{Proposition \ref{rank_drop}.} \emph{Suppose that all the entries of $(-1)^n\De^0$ are strictly positive, and that there exists $z\in S^M\subset \R^n$ and $(x^1,y^1),\dots, (x^n,y^n)$  satisfying, for each $1\leq k\leq n$: 
 $$\begin{cases} x^k\in \mathrm{Ker} \transp{(M^k_0+z^1 M^k_1+\dots+z^n M^k_n)}, &  x^k>0\\
y^k\in \mathrm{Ker}(M^k_0+z^1 M^k_1+\dots+z^n M^k_n), &  y^k> 0\end{cases}$$
Then $z\in S^R$.}\\ 

Our proof relies on two lemmas: the first one comes from Muhi\v{c} and Plestenjak \cite[Lemma 3.4]{MP09}; the second is borrowed from Atkinson \cite[Chapter 6]{atkinson72}. We include both proofs for completeness, as the results are stated slightly differently in \cite{MP09} and \cite{atkinson72}.

\paragraph{Lemma C1.} \emph{
Let $A,B$ be two square matrices of the same size $m$, and let $v\in \R$ be such that $\det(A + v B) = 0$. Suppose there exists $x\in \mathrm{Ker}(\transp{(A + v B)})$ and $y\in \mathrm{Ker}(A + v B)$ such that $\transp{x} B y \neq 0$. Then
$ \mathrm{rank}(A + v  B) <  \max_{w \in \R} \mathrm{rank}(A + wB)$. }\\

\begin{proof} Let $r:= \mathrm{rank}(A + v  B)$ and suppose that $r=\max_{w \in \R} \mathrm{rank}(A + wB)$.  
Note that $r<m$, as the kernel of $A + v B$ contains at least one non-zero vector. 
As the rank of a matrix is the size of the largest invertible square sub-matrices,  $A+v B$ admits some invertible $r\times r$ sub-matrix. By the continuity of the determinant, there exists $\ep>0$ such that this sub-matrix is invertible in the interval $(v-\ep,v+\ep)$, so that $\mathrm{rank}(A + wB)\geq r$ in this interval. As we have supposed that $r$ is the maximal rank, the converse inequality also holds, so that $\mathrm{rank}(A + wB)=r$ on $(v-\ep,v+\ep)$. This implies 
the existence of a vector $y(w)\in \R^m$ with polynomial entries satisfying $(A+w B)y(w)=0$ on $(v-\ep,v+\ep)$ and $y(v)=y$. Derivating the first equality with respect to $w$, one obtains: 
$$ B y(w)+ (A+z B)y'(w)=0$$
Multiplication by $\transp{x}$ and taking $w=v$ yields then:
$$ \transp{x} B y+ \transp{x}(A+v B)y'(v)=0$$
where $\transp{x}(A+v B)=0$ by the choice of $x$. Hence $\transp{x} B y=0$, a contradiction.

\end{proof}

\paragraph{Lemma C2.} \emph{
 Let $z\in S^M$ and let $y^k\neq 0$ belong to $\mathrm{Ker}(M^k_0+z^1M^k+\dots+z^n M^k)$, for all $1\leq k\leq n$. Then $z\in S^\De$ and $(y^1\otimes\dots\otimes y^n) \in \mathrm{Ker}(\De^k-z^k \De^0)$ for all $1\leq k\leq n$.}\\ 

\begin{proof}
Let $z=(z^1,\dots, z^n)\in S^M$. The existence of $y^k\neq 0$ such that $y^k \in \mathrm{Ker}(M^k_0+z^1M^k_1+\dots+z^n M^k _n)$ for all $1\leq k\leq n$ follows from the fact that the matrices $M^k_0+z^1M^k_1+\dots+z^n M^k _n$ are singular. Moreover, one has $y^1\otimes\dots\otimes y^n \neq 0$, as $A\otimes B=0$ if and only if either $A=0$ or $B=0$. Fix $1\leq k\leq n$ and let $\ \widehat{}\ $ the omission of the $k$-th column. Then:
\begin{eqnarray*}
\De^k (y^1\otimes\dots\otimes y^n) &=& (-1)^k \det\nolimits_{\otimes} 
\begin{pmatrix}
M_0^1 y^1 & \dots &\widehat{M^1_k y^1} & \dots  & M^1_n y^1\\
\vdots &     &\vdots  &   & \vdots  \\
M_0^n y^n  &\dots &\widehat{M^n_k y^n} & \dots  & M^n_n y^n \end{pmatrix} \\
&=& 
(-1)^{k+1} \sum_{\ell=1}^n z^\ell \det\nolimits_{\otimes} \begin{pmatrix}
M_\ell^1 y^1 & \dots &\widehat{M^1_k y^1} & \dots  & M^1_n y^1\\
\vdots &     &\vdots  &   & \vdots  \\
M_\ell^n y^n  &\dots &\widehat{M^n_k y^n} & \dots  & M^n_n y^n\end{pmatrix}\\ 
&=& (-1)^{k+1} z^k \det\nolimits_{\otimes} \begin{pmatrix}
M_k^1 y^1 & \dots &\widehat{M^1_k y^1} & \dots  & M^1_n y^1\\
\vdots &     &\vdots  &   & \vdots  \\
M_k^n y^n  &\dots &\widehat{M^n_k y^n} & \dots  & M^n_n y^n\end{pmatrix} \\[0.25cm]
&=& z^k \De^0 (y^1\otimes\dots\otimes y^n)
\end{eqnarray*}
Indeed, the first equality follows from $(K2)$, the second is a consequence of the equalities $M^{\ell'}_0 y^{\ell'}= -\sum_{\ell=1}^n z^\ell M^{\ell'}_\ell y^{\ell'}$ which hold for all $1\leq \ell' \leq n$ and $(K3)$, the third follows from the fact that, for all $\ell\neq k$, the array of matrices has two equal columns so that its Kronecker determinant vanishes, and finally the last equality is obtained by taking a cyclic permutation of the columns (of matrices) which has signature $(-1)^{k+1}$. Hence 
$$(\De^k-z^k \De^0)(y^1\otimes\dots\otimes y^n)=0$$ or, equivalently, 
$y^1\otimes\dots\otimes y^n\in \mathrm{Ker}(\De^k-z^k\De^0)$ and $\det(\De^k-z^k\De^0)=0$. The result follows as this holds for every $1\leq k\leq n$. 
\end{proof}

\bigskip 

We are now ready to prove Proposition \ref{rank_drop}. For any matrix $M$, we write $M>0$ to indicate $M\geq 0$ and $M\neq 0$. 
\paragraph{Proof of Proposition \ref{rank_drop}.}On the one hand, ${x}^1 \otimes \cdots \otimes {x}^n>0$ 
since for any pair of matrices $A,B>0$ implies $A\otimes B> 0$. Similarly,
${y}^1 \otimes \cdots \otimes {y}^n>0$. 
Together with the assumption that all entries of $(-1)^n\De^0$ are strictly positive, it follows that: $$\transp{({x}^1 \otimes \cdots \otimes {x}^n)}\De^0 (y^1 \otimes \cdots \otimes {y}^n)\neq0$$ 
On the other hand, by Lemma C2, one has $z\in S^\De$ and ${y}^1 \otimes \cdots \otimes {y}^n\in \mathrm{Ker}({\De}^k-z^k{\De}^0)$ for all $1\leq k\leq n$. Reversing the roles of the players, one similarly has ${x}^1 \otimes \cdots \otimes {x}^n \in \mathrm{Ker}(\transp{({\De}^k-z^k{\De}^0))}$ for all $1\leq k\leq n$. 
The result follows then from Lemma C1, applied to 
$\De^k$, $\De^0$, $-z^k$, ${x}^1 \otimes \cdots \otimes {x}^n$ and $y^1 \otimes \cdots \otimes {y}^n$, for all $1\leq k\leq n$.$ \hfill \blacksquare$

\section*{Acknowledgements} 
The authors are very grateful to the comments and insight brought by the Editor and the two anonymous referees. \\ 

The second author gratefully acknowledges the support of the French National Research Agency, under
grant ANR CIGNE (ANR-15-CE38-0007-01).

\bibliographystyle{amsplain}
\bibliography{bibliothese2}

\end{document}